\documentclass[11pt]{article}

\usepackage[margin=1in]{geometry}
\usepackage{amsmath,amsthm,amssymb,amsfonts,float}
\usepackage{listings}

\makeatletter
\newcommand{\subjclass}[2][2010]{%
  \let\@oldtitle\@title%
  \gdef\@title{\@oldtitle\footnotetext{#1 \emph{Mathematics subject classification.} #2}}%
}
\makeatother

\usepackage[hidelinks]{hyperref}
\usepackage[normalem]{ulem}
\newcommand\redsout{\bgroup\markoverwith{\textcolor{red}{\rule[0.5ex]{2pt}{0.4pt}}}\ULon}

\usepackage{color}

\usepackage{esint}

\usepackage{mathtools,enumitem,mathrsfs}

\usepackage[compact]{titlesec}

\usepackage{comment}

\mathtoolsset{showonlyrefs=true}
\usepackage[nocompress]{cite}
    
\newcommand{\leqc}{\lesssim}

\newcommand{\grad}{\nabla}

\newcommand{\Qt}{\mathcal{Q}}

\newcommand{\one}{\mathbf{1}}
\newcommand{\ccdot}{\;\cdot\;}

\newcommand{\R}{\mathbb{R}}

\newcommand{\N}{\mathbb{N}}
\newcommand{\T}{\mathbb{T}}

\newcommand{\Pt}{\mathcal{P}}

\newcommand{\dee}{\mathrm{d}}
\newcommand{\ds}{\dee s}
\newcommand{\dt}{\dee t}

\newcommand{\dx}{\dee x}

\newcommand{\dy}{\dee y}

\usepackage{bbm}

\renewcommand{\P}{\mathbf{P}}

\newcommand{\E}{\mathbf{E}}

\theoremstyle{plain}
\newtheorem{theorem}{Theorem}

\newtheorem{proposition}{Proposition}[section]
\newtheorem{assumption}{Assumption}
\newtheorem{lemma}[proposition]{Lemma}

\newtheorem{definition}{Definition}[section]
\newtheorem{corollary}{Corollary}
\theoremstyle{definition}
\newtheorem{remark}{Remark}[section]

\usepackage{todonotes}

\numberwithin{equation}{section}
    
\begin{document}

\title{A pathwise approach to the enhanced dissipation of passive scalars advected by shear flows}
\author{Victor Gardner \and Kyle L. Liss \and Jonathan C. Mattingly
} 
\maketitle

\begin{abstract}
We develop a framework for studying  the enhanced dissipation of passive scalars advected by shear flows based on analyzing the particle trajectories of the stochastic differential equation associated with the governing drift-diffusion equation. We consider both shear flows on $\T^2$ and radially symmetric shears on $\R^2$ or the unit disk. Using our probabilistic approach, we are able to recover the well-known enhanced dissipation timescale for smooth shear flows on $\T^2$ with finite-order vanishing critical points \cite{BCZ17, ABN22, wei2021diffusion, Villringer24} and a generalized version of the results for radially symmetric shear flows from \cite{CZD20}. We also obtain results for shear flows with singularities and critical points where the derivative vanishes to infinite order. The proofs are all based on using Girsanov's theorem to reduce enhanced dissipation to a quantitative control problem that can be solved by leveraging the shearing across streamlines. Our method also has the feature that it is local in space, which allows us also to obtain estimates on the precise decay rate of solutions along each streamline in terms of the local shear profile.   
\end{abstract}

\setcounter{tocdepth}{3}
{\small\tableofcontents}

\section{Introduction}

Consider a passive scalar $f$, defined on a two-dimensional spatial
domain $D$, that is advected by an autonomous, divergence-free
velocity field $u\colon D \to \R^2$ and simultaneously undergoes molecular diffusion. The dynamics of $f$ are governed by the advection-diffusion equation
    \begin{equation} \label{eq:ADE}
    \begin{cases}
        \partial_t f + u \cdot \grad f = \nu \Delta f, \\ 
        f|_{t=0} = f_0,
    \end{cases}
    \end{equation}
    where $\nu > 0$ is a diffusivity parameter and $f_0\colon D \to \R$ is a mean-free initial data. In this paper, we study the quantitative decay rate of solutions to \eqref{eq:ADE} for $0< \nu \ll 1$ in the cases where $u$ is a shear flow on the two-dimensional torus $\T^2 \cong [0,1)^2$ and a radially symmetric shear on either $\R^2$ or the unit disk $\mathbb{D} = \{\mathbf{x} \in \R^2: |\mathbf{x}| \le 1$\}. That is, we consider velocity fields of the form
    \begin{equation} \label{eq:b(y)ex}
    u\colon\T^2 \to \R^2, \quad u(x,y) = b(y)\hat{e}_x 
    \end{equation}
    and 
    \begin{equation}\label{eq:v(r)etheta} 
    u\colon\R^2\text{ or }\mathbb{D} \to \R^2, \quad u(r,\theta) = r v(r) \hat{e}_\theta,
    \end{equation}
    where $b\colon\T \to \R$, $v\colon I \to \R$ with $I$ denoting either
    $[0,\infty)$ or $[0,1]$, and $(r,\theta) \in I \times [0,2\pi)$
    are the standard polar coordinates. Here $\hat{e}$, with the appropriate
    subscript, denotes the unit vector in the specified coordinate direction. 
    When $D$ is the unit disk, \eqref{eq:ADE} is supplemented with homogeneous Neumann boundary conditions.
    
   In each of the settings above, the average of the solution along the streamlines of $u$ decouples from the rest of the dynamics and simply solves a one-dimensional diffusion equation. For example, when $D = \T^2$ and $u$ is given by \eqref{eq:b(y)ex}, the streamline average 
    $$ \langle f \rangle (t,y) : = \int_\T f(t,x,y) \dx  $$
    is a mean-zero solution of
    \begin{equation} \label{eq:1dheat}
    \partial_t \langle f \rangle = \nu \partial_{yy} \langle f \rangle 
    \end{equation}
    on $\T$, and hence decays to zero only on the slow $\nu^{-1}$ diffusive timescale. On the other hand, it is known that the interaction between the velocity differential across streamlines and the diffusion can cause solutions of \eqref{eq:ADE} to homogenize \textit{along streamlines} (i.e., for $f - \langle f \rangle$ to decay) on a timescale much faster than $\nu^{-1}$ when $\nu \ll 1$ \cite{Bajer01, Rhines83, ECZM19, BCZ17, BCZE22}. This phenomenon is typically referred to as \textit{enhanced dissipation}. Precisely quantifying such enhanced decay rates in terms of the shearing properties of $u$ is the problem that we are interested in here.

    The dissipation enhancement of passive scalars advected by shear flows of the form \eqref{eq:b(y)ex} and \eqref{eq:v(r)etheta} has been studied extensively in the literature, primarily from the viewpoints of hypocoercivity \cite{BCZ17, CZD20, CZ20}, quantitative hypoelliptic smoothing \cite{ABN22}, and spectral theory \cite{wei2021diffusion, He22, CZCW20}. The case of smooth shear flows was thoroughly investigated in \cite{BCZ17, wei2021diffusion,ABN22}, where it was shown that if $b\colon\T \to \R$ possesses at most finitely many critical points and $b'$ vanishes to finite order at each of them, then
    \begin{equation} \label{eq:introdecay}
        \|f(t) - \langle f\rangle(t)\|_{L^2} \leqc  e^{-\frac{t}{T_\nu}}\|f_0 - \langle f_0 \rangle\|_{L^2},
    \end{equation}
  where $T_\nu \approx \nu^{-\frac{N+1}{N+3}}$ and $N \in \N$ denotes the maximal order of vanishing of $b'$. In other words, $f$ converges to its streamline average on the timescale $T_\nu$, which is much smaller than the $\nu^{-1}$ diffusive timescale as $\nu \to 0$.  Notice that here $T_\nu \gtrsim \nu^{-1/2}$, as a smooth function $b\colon \T \to \R$ necessarily has at least one critical point. This decay timescale can be improved upon for rough shears. Indeed, it is known that for any $\alpha \in (0,1)$ there exists $b\colon \T \to \R$ that is sharply $\alpha$-H\"{o}lder continuous at each point and such that \eqref{eq:introdecay} holds with $T_\nu \approx \nu^{-\frac{\alpha}{\alpha+2}}$ \cite{wei2021diffusion,CZCW20}. For radially symmetric shears of the form \eqref{eq:v(r)etheta}, there is the work \cite{CZD20} that studies \eqref{eq:ADE} on $\R^2$ in the case where $v(r) = r^q$ for some $q \ge 1$. Using an adaptation of the methods from \cite{BCZ17}, it was shown that solutions of \eqref{eq:ADE} converge to their $\theta$-average 
\begin{equation} \label{eq:thetaaverage}
    \langle f \rangle(t,r):=\frac{1}{2\pi}\int_0^{2\pi} f(t,r,\theta)\dee \theta
\end{equation}
on a timescale $T_\nu \leqc (1+|\log \nu|)\nu^{-\frac{q}{q+2}}$.
    
    Our goal in this work is to revisit the problem of shear flow enhanced dissipation from a pathwise perspective based on studying the random particle trajectories associated with \eqref{eq:ADE} and viewing the convergence rate to equilibrium in terms of the corresponding Markov process. In the setting where $u(x,y) = b(y)\hat{e}_x$ for $b\colon\T \to \R$, \eqref{eq:ADE} becomes 
    \begin{equation} \label{eq:ADEbshearintro}
        \begin{cases}
            \partial_t f + b(y)\partial_x f = \nu \Delta f, \\
            f|_{t=0} = f_0
        \end{cases}
    \end{equation}
    and the corresponding Markov process is generated by the It\^{o} stochastic differential equation (SDE)
\begin{equation}\label{eq:SDEintro}
\begin{cases}
        \dee x_t = -b(y_t)\dt + \sqrt{2\nu}\dee W_t, \\ 
        \dee y_t = \sqrt{2\nu} \dee B_t.
\end{cases}
    \end{equation}
Here, $W_t$ and $B_t$ are two independent, standard Brownian motions on $\T$ defined on a common probability space $(\Omega, \mathcal{F},\P)$. The solution of \eqref{eq:ADEbshearintro} is then given by $f(t,x,y) = \E f_0(x_t,y_t)$, where $(x_t,y_t)$ solves \eqref{eq:SDEintro} with initial condition $(x_0,y_0) = (x,y) \in \T^2$ and $\E$ denotes averaging with respect to the probability measure $\P$. Equivalently, the solution may be written as 
\begin{equation} \label{eq:ProbRep2}
f(t,x,y) = \int_{\T^2} f(t,x',y')\Pt_t(x,y,\dee x', \dee y'), 
\end{equation}
where $\Pt_t\colon\T^2 \times \mathcal{B}(\T^2) \to [0,1]$ is the Markov transition kernel associated with \eqref{eq:SDEintro} defined by 
\begin{equation}\label{eq:kernelintro}
    \Pt_t(x,y,A) = \P\big((x_t,y_t) \in A\text{ }\big|\text{ }(x_0,y_0) = (x,y)\big).
\end{equation}
Note that $\Pt_t(x,y,\ccdot)$ is simply the law of the solution of \eqref{eq:SDEintro} with initial condition $(x,y)$.

To see how one may think about enhanced dissipation probabilistically, observe that by \eqref{eq:ProbRep2}, the solution of \eqref{eq:ADEbshearintro} satisfies
\begin{align}
    \left|f(t,x,y) - \int_\T f(t,\tilde{x},y) \dee \tilde{x}\right| &= \left|\int_\T \left( \int_{\T^2}f(t,x',y')[\Pt_t(x,y,\dee x', \dee y') -\Pt_t(\tilde{x},y,\dee x', \dee y')]\right)\dee \tilde{x} \right|\nonumber \\ 
    &\le \|f_0\|_{L^\infty} \sup_{x,\tilde{x},y \in \T}\|\Pt_t(x,y,\ccdot) - \Pt_t(\tilde{x},y,\ccdot)\|_{TV}, \label{eq:streamlineheuristic}
\end{align}
where $\|\mu_1 - \mu_2\|_{TV}$ denotes the total variation distance
between two probability measures $\mu_1$ and $\mu_2$.\footnote{Throughout this
note, we use the probabilist's normalization of the TV-distance which
yields distance between 0 and 1 inclusively.} The calculation in \eqref{eq:streamlineheuristic} says that the speed
of homogenization along streamlines is determined by the rate at which
the probability laws of two solutions of \eqref{eq:SDEintro} with
initial conditions on the same streamline converge to each other in
the total variation metric. Heuristically, one expects this
convergence to occur on a timescale comparable to the time it takes
the variance of $x_t$ to become $\nu$-independent. From this
viewpoint, enhanced dissipation becomes possible because the shearing
interacts with the random fluctuations across streamlines to cause the
variance of $x_t$ to grow much faster than it would from the effects
of Brownian motion alone. More precisely, suppose that there exists a strictly increasing, continuous function $\varphi:[0,h_0] \to (0,\infty)$ such that for every $y \in \T$ there is a choice $\iota(y) \in \{1,-1\}$ so that
\begin{equation} \label{eq:varphiintro}
    |b(y+\iota(y) h) - b(y)| \ge \varphi(h) \quad \forall 0<h\le h_0.
\end{equation}
Assuming the minimal velocity differential across streamlines described by \eqref{eq:varphiintro}, one predicts by solving \eqref{eq:SDEintro} to have a lower bound on the variance of $x_t$ that scales in $t$ and $\nu$ like 
\begin{equation} \label{eq:VarHeurstic}
\sqrt{\mathrm{Var}(x_{t})} \gtrsim \int_0^{t} |b(y_0 + \iota(y_0) \sqrt{2\nu s}) - b(y_0)|\ds \gtrsim t \varphi(\sqrt{\nu t}). 
\end{equation}
Given \eqref{eq:varphiintro}, it is natural then to hope that solutions of \eqref{eq:ADEbshearintro} converge to their streamline average on a timescale $T_\nu$ defined by
\begin{equation} \label{eq:Tnu}
    T_\nu = \inf\{t \ge 0: t \varphi(\sqrt{\nu t}) = 1\},
\end{equation}
which one can easily check always satisfies $\lim_{\nu \to 0} \nu T_\nu = 0$ (see Remark~\ref{rem:Tnu3}). We make the heuristics above rigorous with an appropriate application of Girsanov's theorem that reduces the estimate
\begin{equation}
    \|\Pt_T(x,y,\ccdot) - \Pt_T(\tilde{x},y,\ccdot)\|_{TV} \le 1-\epsilon
\end{equation}
for some $T \leqc T_\nu$ and $\epsilon \in (0,1)$ to a quantitative control problem that can be solved by leveraging the shearing across streamlines. 

In the $\T^2$ setting, under a fairly general assumption on $b\colon \T \to \R$ that ensures \eqref{eq:varphiintro} is satisfied for some $\varphi$, we prove in Theorem~\ref{thm:T2general} that \eqref{eq:introdecay} indeed holds with $T_\nu$ as defined in \eqref{eq:Tnu}. As simple corollaries of Theorem~\ref{thm:T2general}, we recover the results from \cite{BCZ17,wei2021diffusion,ABN22} for smooth shears on $\T^2$ with critical points vanishing to finite order, but are also able to cover examples with critical points vanishing to infinite order and shears with finitely many points where $b'$ becomes singular. For radial shear flows, we obtain sharp enhanced dissipation estimates in a general smooth setting on $\R^2$ and the unit disk $\mathbb{D}$ that contains the setting considered in \cite{CZD20}. With just minor modifications to the argument, we can also treat singular velocity fields on $\mathbb{D}$ of the form $r^{-\alpha} \hat{e}_\theta$ for $0 < \alpha \le 1$. The case $\alpha = 1$ corresponds to the velocity field generated by a point vortex stationary solution of the 2d Euler equation, and hence is of natural physical interest. For precise statements of our results in the radially symmetric setting, see Theorem~\ref{thm:RadialGen} and Remark~\ref{rem:vortex}. The proofs are essentially the same as they are for shears on $\T^2$, and in this sense we view our probabilistic approach as quite robust. 

Besides allowing us to obtain enhanced dissipation estimates in a wide range of situations with essentially a single proof, our framework has the additional feature that it naturally treats the convergence to equilibrium in a streamline-by-streamline manner. More specifically, in view of \eqref{eq:streamlineheuristic} and the first inequality in \eqref{eq:VarHeurstic}, the time $T_\nu$ defined in \eqref{eq:Tnu} naturally generalizes to a $y$-dependent timescale that determines the expected rate of homogenization along the streamline at height $y$ in terms of the local velocity differential. As a consequence of this, our proofs imply as relatively easy corollaries estimates that quantify the local decay rate of solutions on each streamline; see Theorem~\ref{thm:T2local} for a precise statement. Such localized estimates are new to the best of our knowledge and while they can likely be obtained with an adaptation of previous proofs of enhanced dissipation, the probabilistic method has the advantage that they appear quite naturally.

While preparing the present manuscript, the work \cite{Villringer24} appeared that also considers shear flow enhanced diffusion from a stochastic viewpoint. The author works in the setting of smooth shears on $\T^2$ and reproduces the results of \cite{BCZ17,ABN22,wei2021diffusion} using the Malliavin calculus. The proof hinges on the fact that the interaction between the shearing and diffusion causes one of the eigenvalues of the Malliavin covariance matrix to grow at a rate much faster than $\nu t$. This arises from the same heuristics described above and captures the content of \eqref{eq:VarHeurstic}, but from an ``infinitesimal'' viewpoint. One can thus view \cite{Villringer24} as treating the special case of smooth shears with critical points vanishing to finite order with an infinitesimal control argument, rather than reducing to a control argument for finite perturbations using Girsanov's theorem as we do here. A stochastic viewpoint has also been used previously in \cite{CZDrivas19} to prove \textit{upper} bounds on enhanced dissipation rates for shear flows in various settings, but this is quite different from the ideas used in this paper and \cite{Villringer24} to obtain \textit{lower} bounds.

We should conclude this introductory discussion by remarking that the study of enhanced dissipation extends well beyond the setting of passive scalars advected by autonomous, 2d shear flows considered here. Without any attempt to be exhaustive, we mention that enhanced dissipation results have been established for passive scalars transported by shear flows in higher dimensions \cite{CZG21}, general autonomous Hamiltonian flows in 2d \cite{Vukadinovic21, BCZE22, DolceHamiltonian24}, and incompressible flows assumed to satisfy suitable dynamical conditions (e.g., a mixing rate for the measure preserving $\nu = 0$ system)  \cite{FengIyer19,ECZM19, CKRZ08, Zlatos10}. Enhanced dissipation has also been proven for solutions of the linearized and fully nonlinear 2d Navier-Stokes equations in the vicinity of certain planar shear flows \cite{WeiZhangKolmogorov20, WeiKolmogorov19, ElgindiPoiseuille20, DelZotto23, BeekieDepletion24, BGM17, BVW18, WeiZhangCouette, MasmoudiZhao2dCouette} and exact vortex solutions \cite{GallayVortex18, WeiZhangLiOseen17}. Lastly, there has been significant recent progress in constructing time-dependent velocity fields such that solutions of the associated advection diffusion decay at an optimal rate much faster than the polynomial-in-$\nu$ enhanced dissipation timescales induced by shear flows \cite{ELM23, IyerCooperman24,CZabc24, BBPSenhanced, Zworski23}. Regarding extensions of the ideas in the present paper, the framework we use here is certainly not limited to shear flows and it would be natural to next study possible applications in the setting of general autonomous Hamiltonian flows in two dimensions.

\subsection*{Organization of the paper}

The remainder of this paper is organized as follows. In Section~\ref{sec:results}, we formulate precisely our assumptions and state our main results. Section~\ref{sec:probability} is devoted to making rigorous our probabilistic framework that reduces enhanced dissipation to a quantitative control problem. In Sections~\ref{sec:T2proof} and~\ref{sec:Radialproofs}, we solve the relevant control problems on $\T^2$ and $\R^2$ (or the unit disk), respectively, to complete the proofs of our main theorems. We conclude in Section~\ref{sec:local} with the arguments necessary to upgrade our results to local enhanced dissipation estimates. 

\vspace{0.2cm}

\noindent \textbf{Acknowledgements:} J.M and K.L thank Theo Drivas for suggesting to study enhanced dissipation from a probabilistic perspective and the Simons Center for Geometry and Physics at Stony Brook University for hospitality during the conference ``Small Scale Dynamics in Fluid Motion,'' where initial discussions on this project began. The authors also thank Tarek Elgindi and Raj Beekie for helpful comments on a first draft of this paper and the NSF for its support through the RTG grant DMS-2038056.

\section{Assumptions and main results} \label{sec:results}

In this section we state precisely our assumptions and main results. We begin with the $\T^2$ case and then discuss our results in the radially symmetric setting.

\subsection{Results for shears on $\T^2$} \label{sec:T2results} 

Here we present our assumptions and results for shear flows $u\colon \T^2 \to \R^2$ of the form
\begin{equation} \label{eq:T2shear}
u(x,y) = 
\begin{pmatrix}
    b(y) \\ 0
\end{pmatrix}
,
\end{equation} 
where $b\colon  \T \to \R$. Recall that in this case, \eqref{eq:ADE} becomes 
\begin{equation} \label{eq:ADEbshear}
\begin{cases}
    \partial_t f + b(y)\partial_x f = \nu \Delta f,\\ 
    f|_{t=0} = f_0.
\end{cases}
\end{equation}
As mentioned earlier, the average of the solution along streamlines just solves the one-dimensional heat equation, and so we will always assume that the initial data for \eqref{eq:ADEbshear} satisfies 
\begin{equation}\label{eq:T2meanfree}
    \int_\T f_0(x,y)\dx = 0 \quad \forall y\in \T,
\end{equation}
which is propagated by the evolution of \eqref{eq:ADEbshear}. To simplify the statements of our main results in this setting, we introduce a precise definition of enhanced dissipation which states essentially that solutions of \eqref{eq:ADEbshear} decay on a timescale that is $o(\nu^{-1})$ as $\nu \to 0$.
\begin{definition}[Dissipation
  Enhancing] \label{def:ED}
    We say that a shear flow of the form \eqref{eq:T2shear} is dissipation enhancing with rate function $\lambda\colon (0,1) \to (0,\infty)$ if $$\lim_{\nu \to 0} \frac{\nu}{\lambda(\nu)} = 0$$
    and there exist constants $C, c > 0$ such that for all $f_0 \in L^2(\T^2)$ satisfying \eqref{eq:T2meanfree} the solution of \eqref{eq:ADEbshear} obeys
    \begin{equation} \label{eq:EDdef}
        \|f(t)\|_{L^2} \le C e^{-c\lambda(\nu)t}\|f_0\|_{L^2}.
    \end{equation}
\end{definition}

\subsubsection*{Enhanced dissipation under a general assumption}

We begin with a theorem that gives enhanced dissipation and a
quantitative lower bound of $\lambda(\nu)$ under a fairly general
assumption on $b$. A range of applications will be given in the next section. Roughly speaking, we require $b$ to possess finitely many critical points and be smooth away from a discrete set of points where $b'$ has a suitable discontinuity. Some additional conditions are also imposed in the vicinity of points where $b'$ vanishes to infinite order or blows up. Precisely, we make the following assumption.

\begin{assumption} \label{Ass:T2general}
  The function $b\colon\T \to \R$ is continuous, smooth away from a discrete set of points $\{\bar{y}_i\}_{i=1}^{N} \subseteq \T$, and possesses at most finitely many critical points $\{y_i\}_{i=1}^{M} \subseteq \T$. Moreover: 
  \begin{itemize}
      \item[(a)] For each $1 \le i \le N$, the one-sided limits
      $\lim_{y\to \bar{y}_i^\pm} |b'(y)|$ both exist and are valued in $(0,\infty].$ In the case that a one-sided limit is infinite, its convergence is eventually monotone and there exists $C > 0$ such that for small positive $h$ the function $F(h) = b(\bar{y}_i\pm h) - b(\bar{y}_i)$ satisfies 
      $$|F(h)| \le Ch|F'(h)| \quad \text{and} \quad |F''(h)| \le C \frac{|F'(h)|}{h}.$$
      \item[(b)] There exists $C \ge 1$ such that if $b'$ vanishes to infinite order at some $y_i$ (i.e., $b^{(k)}(y_i) = 0$ for all $k$), then the convergence of both the one-sided limits $\lim_{y\to y_i^{\pm}} |b'(y)| = 0$ is eventually monotone and the function $F(h) = b(y_i+h)-b(y_i)$ satisfies
      $$\limsup_{h \to 0} \frac{\sqrt{|F(h)|}}{ |h| |F'(Ch)|} < \infty.  $$
  \end{itemize}
\end{assumption}

\begin{remark}
    In the well-studied case where $b$ has only finitely many critical points and $b'$ vanishes to finite order at each of them, then (a) and (b) are vacuously true. In the event that $b'$ becomes singular or critical points vanishing to infinite order exist, then the technical estimates of (a) and (b) above are made primarily to rule out pathological behaviors and are easily seen to be true in ``typical'' examples. The bounds in (a) say essentially that each blow up of $b'$ resembles the behavior of $x\mapsto x^p$ with $p \in (0,1)$ near $x = 0$, and are true for instance at any $\bar{y}_i$ where 
    $$b(\bar{y}_i + h) - b(\bar{y}_i) \propto |h|^\alpha \log^\beta(|h|^{-1})$$ for some $\alpha \in (0,1)$ and $\beta \in \R$. Regarding (b), this assumption is true for common functions that vanish to infinite order, and holds for example at any $y_i$ where
    $$ b(y_i+h) - b(y_i) \propto e^{-\frac{1}{|h|^p}} $$
    for some $p > 0$.
\end{remark}

Since Assumption~\ref{Ass:T2general} ensures that the set of points where $b'$ either vanishes or has a discontinuity is finite, there exists some $h_0 > 0$ with the property that for any $y \in \T$ there is a choice $\iota(y) \in \{1,-1\}$ such that no elements of $S = \{\bar{y}_1, \ldots, \bar{y}_N, y_1, \ldots, y_M\}$ lie strictly between $y$ and $y+\iota h_0$. If $y$ is far enough from the set $S$, then $\iota(y)$ is not necessarily unique, but if, for instance, $y$ is sufficiently close to a critical point $y_i$, then $\iota(y)$ is the particular choice such that $h \mapsto y + \iota(y)h$ moves away from $y_i$. As $b'$ can only change sign at the points in $S$, the definition of $\iota(y)$ implies that $h \mapsto |b(y+\iota(y)h) - b(y)|$ is strictly increasing for $h \in [0,h_0]$. In particular, $b$ possesses a quantifiable minimal velocity differential across streamlines in the sense that there is a monotone increasing function $\varphi \in C([0,h_0];\R_{\ge 0})$ with $\varphi(0) = 0$ and such that
\begin{equation}\label{eq:varphi}
  |b(y+\iota(y) h) - b(y)| \ge \varphi(h) \quad \forall h\in[0,h_0].
  \end{equation}
As discussed in the introduction, it is natural to expect the enhanced dissipation rate of \eqref{eq:ADEbshear} to be determined by the growth rate of $\varphi$ through the relation~\eqref{eq:Tnu}. This is made precise by our first main result, which is stated as follows.

\begin{theorem} \label{thm:T2general}
    Suppose that $b\colon  \T \to \R$ satisfies Assumption~\ref{Ass:T2general}. Let $\varphi$ be as in \eqref{eq:varphi} and let 
    $$ T_\nu = \inf\{t\ge0: t \varphi(\sqrt{\nu t}) = 1\}. $$
   Then, $u(x,y) = (b(y),0)$ is dissipation enhancing with rate function $\lambda(\nu) = T_\nu^{-1}$.
\end{theorem}

\begin{remark}\label{rem:Lp}
    The choice of $L^2$ in Definition~\ref{def:ED} instead of some other $L^p$ norm was irrelevant. Because our argument is based on \eqref{eq:streamlineheuristic}, our proof of Theorem~\ref{thm:T2general} actually gives
    \begin{equation}\label{eq:Lpremark}
    \|f(t)\|_{L^\infty} \le C e^{-c\lambda(\nu)t}\|f_0\|_{L^\infty}.
    \end{equation}
    Since the adjoint equation of \eqref{eq:ADEbshear} is the same except for changing the sign of $b$, by duality \eqref{eq:Lpremark} then also holds with $\|\cdot\|_{L^\infty}$ replaced by $\|\cdot\|_{L^1}$. By interpolation, we thus obtain enhanced dissipation in $L^p$ for any $p \in [1,\infty]$ with a decay rate that is uniform in $p$. This should be contrasted with hypocoercivity based proofs of enhanced dissipation, which don't clearly work in spaces that are not $L^2$ based. 
\end{remark}

\begin{remark}\label{rem:hypoelliptic}
    Theorem~\ref{thm:T2general} is also true for the hypoelliptic version of \eqref{eq:ADEbshear} given by 
    $$ \partial_t f + b(y)\partial_x f = \nu \partial_{yy}f $$
    and the proof is exactly same. 
\end{remark}

\begin{remark} \label{rem:Tnu3}
    The fact that $\lim_{\nu \to 0} \nu \lambda(\nu)^{-1} = \lim_{\nu \to 0} \nu T_\nu = 0$ is easy to establish. Indeed, for any $\epsilon > 0$, setting $t = \epsilon \nu^{-1}$ we have
    $$ t \varphi(\sqrt{\nu t}) \ge \epsilon \nu^{-1} \varphi(\sqrt{\epsilon})\ge c_\epsilon \nu^{-1} $$
    for some $c_\epsilon > 0$ depending on $\varphi$. Thus, $\nu T_\nu \le \epsilon$ for all $\nu \le c_\epsilon$.   
    While one always has $T_\nu = o(\nu^{-1})$ as $\nu \to 0$, we will see in the corollaries to Theorem~\ref{thm:T2general} below that $T_\nu$ can often be quantified precisely in concrete examples.
\end{remark}

\begin{remark}
    Assumption~\ref{Ass:T2general}b regarding the structure of $b$ at points where $b'$ vanishes to infinite order is not needed to deduce enhanced dissipation. A straightforward modification of the proof of Theorem~\ref{thm:T2general} shows that even if Assumption~\ref{Ass:T2general}b is removed, $u$ is still dissipation enhancing with some rate function $\lambda$. However, in this case the proof does not show that $\lambda$ is bounded below by its definition via \eqref{eq:Tnu}. This is actually a relatively small concession because regardless of whether or not Assumption~\ref{Ass:T2general}b holds, the existence of a critical point where $b'$ vanishes to infinite order implies that the optimal rate function decays with $\nu$ faster than $\nu^p$ for any $p < 1$.
\end{remark}

\subsubsection*{Applications of Theorem~\ref{thm:T2general}}

We now present a few concrete examples where Theorem~\ref{thm:T2general} can be applied and the associated $T_\nu$ is explicitly computable or at least easy to estimate. As mentioned earlier, they will include the now well-understood case of smooth shears with finitely many critical points that each vanish to finite order, but also lesser studied settings where $b'$ may blow up or vanish to infinite order. 

We begin with the case where $b$ is smooth and $b'$ vanishes to finite order at each critical point. For later reference and to establish some notation, we state this as the precise assumption below.
\begin{assumption}\label{Ass:classical}
 The function $b\colon \T \to \R$ is smooth and possesses finitely many critical points $\{y_1,\ldots, y_M\}$. Moreover, there exists $N \in \N$ such that for each $1 \le i \le M$ we have 
 \begin{equation} \label{eq:finitevanish}
     b^{(k)}(y_i) \neq 0 \quad \text{for some} \quad k \le N+1.
 \end{equation}
\end{assumption}
\noindent 
For $b\colon \T \to \R$ satisfying Assumption~\ref{Ass:classical}, we will always denote by $N$ the minimal natural number such that \eqref{eq:finitevanish} holds.

As stated earlier, enhanced dissipation in the setting of Assumption~\ref{Ass:classical} has been studied in \cite{BCZ17, ABN22, wei2021diffusion} and more recently in \cite{Villringer24}. It was first proven in \cite{BCZ17} using a hypocoercivity argument that such a velocity field $u$ is dissipation enhancing with rate function given by $\lambda(\nu) = \nu^{\frac{N+1}{N+3}}(1+|\log \nu|)^{-2}$. The logarithmic correction was removed in \cite{wei2021diffusion} using a proof based on spectral theory and the resulting rate function was shown to be essentially sharp in \cite{CZDrivas19}. Using Theorem~\ref{thm:T2general}, the sharp enhanced dissipation rate for shears satisfying Assumption~\ref{Ass:classical} is easily recovered:

\begin{corollary} \label{cor:finitevanish}
    Let $b\colon  \T \to \R$ satisfy Assumption~\ref{Ass:classical}. Then, $u(x,y) = (b(y),0)$ is dissipation enhancing with rate function $\lambda(\nu) = \nu^{\frac{N+1}{N+3}}$. 
\end{corollary}

\begin{proof}
As $b$ is smooth and does not possess any critical points where $b'$ vanishes to infinite order, Assumption~\ref{Ass:T2general} is automatically satisfied. Moreover, it is easy to see from Taylor's theorem that one may take $\varphi(h) = ch^{N+1}$ for some $c > 0$ sufficiently small. Solving the equation $T_\nu\varphi(\sqrt{\nu T_\nu}) = 1$ shows that $T_\nu = c^{-\frac{2}{N+3}}\nu^{-\frac{N+1}{N+3}}$, and the result then follows by Theorem~\ref{thm:T2general}. 
\end{proof}

We now deviate from the setting of Assumption~\ref{Ass:classical} and
allow for critical points where $b'$ vanishes to infinite order. As
the rate function in Corollary~\ref{cor:finitevanish} degenerates to
the trivial one $\lambda(\nu) = \nu$ in the limit $N \to \infty$, the
enhanced dissipation rate is not immediately clear in this
case. Generally speaking, one expects to have $1/\lambda(\nu) =
o(\nu^{-1})$ as $\nu \to 0$ with $\lambda(\nu) = r(\nu) \nu$ for some
function $r\colon (0,1) \to (0,\infty)$ that depends on the particular rate of vanishing of $b'$ at the critical points and, while blowing up as $\nu \to 0$, satisfies $\lim_{\nu \to 0} \nu^{\epsilon} r(\nu) = 0$ for every $\epsilon > 0$. An interesting question is then if one can provide quantitative estimates for $r$ that depend on the particular structure of $b$ in a natural way. Using Theorem~\ref{thm:T2general}, we can accomplish this for a variety of typical examples. Corollary~\ref{cor:infinitevanish} below is one instance of this.

\begin{corollary} \label{cor:infinitevanish}
Let $b\colon \T \to \R$ be a smooth function with finitely many critical points. Suppose that 
 $\{y \in \T: b^{(k)}(y) = 0 \quad \forall k \in \N\} = \{y_1, \ldots, y_m\} $
and for each $1\le i \le m$ there exist $h_i, p_i > 0$ and $c_i \in \R$ such that 
$$ b(y_i + h) - b(y_i) = c_i \exp\left(-\frac{1}{|h|^{p_i}}\right) \quad \forall h\in [-h_i,h_i]\setminus\{0\}. $$
Then, $u(x,y) = (b(y),0)$ is dissipation enhancing with rate function 
$$ \lambda(\nu) = \nu|\log\nu|^{2/p}, \quad \text{where} \quad p=\max_{1\le i \le m}p_i.$$
\end{corollary}

\begin{proof}
    Let $F_i(h) = b(y_i+h) - b(y_i)$ for $1\le i \le m$. For all $0<|h| \ll 1$ we have 
    $$ F_i'(h) = \frac{\mathrm{sgn}(h) c_i p_i}{|h|^{p_i+1}}\exp\left(-\frac{1}{|h|^{p_i}}\right) \quad \text{and} \quad F_i''(h) = \frac{c_i p_i}{|h|^{p_i+2}}\left(\frac{ p_i}{|h|^{p_i}}- (p_i+1)\mathrm{sgn}(h)\right)\exp\left(-\frac{1}{|h|^{p_i}}\right), $$
    from which the requirements of Assumption~\ref{Ass:T2general}b follow easily. We see then that Assumption~\ref{Ass:T2general} is satisfied with 
    $ \varphi(h) = c\exp(-|h|^{-p})$ for some $c > 0$. A simple computation shows that for $t(\nu) = 2^{2/p}\nu^{-1}|\log \nu|^{-2/p}$ we have
    $$ t(\nu) \varphi\big(\sqrt{\nu t(\nu)}\big) = \frac{c 2^{2/p}}{\sqrt{\nu}|\log \nu|^{2/p}}.$$
    Thus, $\lim_{\nu \to 0}t(\nu) \varphi\big(\sqrt{\nu t(\nu)}\big) = \infty$, which implies that $T_\nu$ as defined by \eqref{eq:Tnu} satisfies $T_\nu \le t(\nu)$ for all $\nu$ sufficiently small. The result now follows by Theorem~\ref{thm:T2general}.
\end{proof}

\begin{remark}
    Of course, the precise estimate on the rate function provided by Theorem~\ref{thm:T2general} is $\lambda(\nu) \ge 1/F_\nu^{-1}(1)$, where $F_\nu(t) = t\varphi(\sqrt{\nu t})$. The bound on $\lambda(\nu)$ that we have stated in Corollary~\ref{cor:infinitevanish} is simply the natural estimate that one obtains from this fact in the sense that 
    $$ \lim_{\nu\to 0} \nu |\log \nu|^{2/q} F_\nu^{-1}(1) = \infty $$
    for any $q < p$.
\end{remark}

We conclude our discussion of Theorem~\ref{thm:T2general} with an application to shear flows which are only H\"{o}lder continuous. In particular, we consider the situation where $b$ is H\"{o}lder continuous, smooth away from finitely many points, and satisfies $|b'(y)| \ge c > 0$ for some constant $c$ everywhere $b'$ is defined. In view of Corollary~\ref{cor:finitevanish}, one expects that in this case $u(x,y) = (b(y),0)$ is dissipation enhancing with rate function $\lambda(\nu) = \nu^{1/3}$, though solutions with initial data localized near a point where $b'$ blows may actually decay on a timescale much faster than $\nu^{-1/3}$. The precise improvement of the local decay rate of solutions in the vicinity of points where $b$ becomes non-smooth will be discussed briefly in the following section (see Remark~\ref{rem:rough}). For now, we show that in this setting, Theorem~\ref{thm:T2general} applies to shear flows that exhibit a large class of derivative singularities and easily gives the expected $\lambda(\nu) = \nu^{1/3}$ rate function. 

\begin{corollary}
    Let $b\colon \T \to \R$ be smooth away from finitely many points $\{\bar{y}_1, \ldots, \bar{y}_N\}$ and such that for some $c > 0$ we have $|b'(y)| \ge c$ wherever $b'$ is defined. Suppose that for each $1 \le i \le N$ there exist $\alpha_i, h_i \in (0,1)$ and $\beta_i, c_i \in \R$ such that 
    $$ b(\bar{y}_i + h) - b(\bar{y}_i) = c_i |h|^{\alpha_i} \log^{\beta_i}(|h|^{-1}) \quad \forall h\in [-h_i,h_i]\setminus \{0\}. $$
    Then, $u(x,y) = (b(y),0)$ is dissipation enhancing with rate function $\lambda(\nu) = \nu^{1/3}$.
\end{corollary}

\begin{proof}
Let $F_i(h) = c_i|h|^{\alpha_i}\log^{\beta_i}(|h|^{-1})$. For all $0<|h| \ll 1$ we have 
$$ F_i'(h) = c_i \mathrm{sgn}(h)|h|^{\alpha_i-1}\log^{\beta_i-1}(|h|^{-1})(\alpha_i \log(|h|^{-1}) - \beta_i) $$
and 
$$ F_i''(h) = c_i |h|^{\alpha_i-2}\log^{\beta_i-2}(|h|^{-1}) \left[\alpha_i(\alpha_i-1)\log^2(|h|^{-1}) - \beta_i(2\alpha_i +1)\log(|h|^{-1}) + \beta_i (\beta_i-1)\right].  $$
The leading order terms in $F_i'(h)$ and $F_i''(h)$ as $h \to 0$ are the ones proportional to $|h|^{\alpha_i-1}\log^{\beta_i}(|h|^{-1})$ and $|h|^{\alpha_i-2}\log^{\beta_i}(|h|^{-1})$, respectively. Using this fact it is easy to see that the requirements of Assumption~\ref{Ass:T2general}a hold. Therefore, $b$ satisfies Assumption~\ref{Ass:T2general} and moreover we may take $\varphi(h) = c h$ for some $c > 0$ sufficiently small.  Solving the resulting equation $T_\nu \varphi(\sqrt{\nu T_\nu}) = cT_\nu^{3/2} \nu^{1/2} = 1$ for $T_\nu$ and applying Theorem~\ref{thm:T2general} completes the proof.
\end{proof}

\subsubsection*{Local enhanced dissipation estimates}

Theorem~\ref{thm:T2general} gives a lower bound on the enhanced
dissipation rate in terms of the \textit{minimal} velocity
differential of $u$ across its horizontal streamlines. One expects the
decay of solutions to be described more precisely by a $y$-dependent
rate function that is determined from the \textit{local} velocity
differential across streamlines. One of the features of our proof of
Theorem~\ref{thm:T2general} is that it is entirely local in $y$, and
as a consequence provides such a description with only a bit of additional effort. In
particular, we define a $y$-dependent timescale $t_\nu \colon \T \to (0,\infty)$ by
\begin{equation}\label{eq:tnuintro}
    t_\nu(y) = \inf\{t \ge 0: t |b(y+\iota(y) \sqrt{\nu t}) - b(y)| \ge 1\},
\end{equation}
which one should view as the natural generalization of $T_\nu$ given by \eqref{eq:Tnu}. Here, $\iota(y) \in \{1,-1\}$ is the choice of sign discussed just before \eqref{eq:varphi}. When $\iota(y)$ is not unique the two possible values of $t_\nu(y)$ obtained by \eqref{eq:tnuintro} are equivalent up to $\nu$-independent constants and we simply define $\iota(y)$ as the choice that results in the smallest value. In the case that $b$ satisfies Assumption~\ref{Ass:T2general} and does not possess any critical points where $b'$ vanishes to infinite order, the proof of Theorem~\ref{thm:T2general} actually shows that there exist $\epsilon,C > 0$ such that for any $K \subseteq \T$ and probability measures $\mu_1,\mu_2$ that have the same $y$-marginal distributions and satisfy $\mu_i(\T \times K) = 1$ we have
\begin{equation} \label{eq:localdecay1}
\|\mu_1\Pt_T - \mu_2 \Pt_T\|_{TV} \le (1-\epsilon)\|\mu_1 - \mu_2\|_{TV}, \quad \text{where} \quad T = C\sup_{y \in K} t_\nu(y).
\end{equation}
Here, $\mu \Pt_t$ denotes the dual action of the Markov semigroup generated by \eqref{eq:SDEintro} on a probability measure $\mu$, and gives the time-$t$ law of the solution of \eqref{eq:SDEintro} with initial condition distributed according to $\mu$ (see the beginning of Section~\ref{sec:probability} for more details). As the $y$-support of $\mu \Pt_t$ remains essentially unchanged on sub-diffusive timescales (up to a small error term; see Remark~\ref{rem:error}), one can hope to iterate \eqref{eq:localdecay1} to obtain a global-in-time estimate describing the precise $y$-dependent decay rate of $\|\Pt_t(x,y,\ccdot) - \Pt_t(\tilde{x},y,\ccdot)\|_{TV}= \|\delta_{(x,y)}\Pt_t - \delta_{(\tilde{x},y)}\Pt_t \|_{TV}$ to zero. In view of \eqref{eq:streamlineheuristic}, such a bound immediately describes the decay of $\|f(t,\cdot,y)\|_{L^\infty(\T)}$ for solutions satisfying \eqref{eq:T2meanfree}. We will present a result of this kind in detail only for shears satisfying Assumption~\ref{Ass:classical}, as in this case the local timescale $t_\nu(y)$ and associated local decay rate $\lambda_\nu(y):=1/t_\nu(y)$ are easy to estimate. 

Let $\{y_1,\ldots, y_M\}$ denote the critical points of a function $b\colon\T \to \R$ that satisfies Assumption~\ref{Ass:classical} and let $k_i \in \N$ be the minimal natural number such that $b^{(k_i+1)}(y_i) \neq 0$. For $1 \le i \le M$, let
\begin{equation}\label{eq:critical_point_rate}
    \lambda_{\nu,i} = \nu^{\frac{k_i+1}{k_i+3}} \quad \text{and} \quad \ell_{\nu,i} = \sqrt{\nu \lambda^{-1}_{\nu,i}}
\end{equation}
be the decay rate and associated diffusive lengthscale corresponding to the critical point $y_i$. It is straightforward to estimate $\lambda_\nu(y) = 1/t_\nu(y)$ and show that 
\begin{equation} \label{eq:tnuAss2}
    \lambda_\nu(y) \approx \bar{\lambda}_\nu(y),
\end{equation}
where
\begin{equation}\label{eq:localrate}
    \bar{\lambda}_\nu(y) = \begin{cases}
        \lambda_{\nu,i} & y \in (y_i - \ell_{\nu,i}, y_i + \ell_{\nu,i}) \text{ for some }1\le i \le M, \\
        \nu^{1/3}|b'(y)|^{2/3} & \text{otherwise}.
        \end{cases}
    \end{equation}
The behavior of $\bar{\lambda}_\nu(y)$ is best understood by observing from Taylor's theorem that for $\ell_{\nu,i} \le |y-y_i| \ll 1$ there holds 
\begin{equation} \label{eq:weightequiv}
   \bar{\lambda}_\nu(y) \approx \nu^{1/3}|y-y_i|^{\frac{2k_i}{3}}.
\end{equation}
Therefore, $\bar{\lambda}_\nu(y)$ describes the manner in which the local decay rate in the vicinity of each critical point grows, as $|y-y_i|$ increases from $\ell_{\nu,i}$ to $\nu$-independent, from the minimal value $\lambda_{\nu,i}$ seen in Corollary~\ref{cor:finitevanish} to the rate $\nu^{1/3}$ for strictly monotone shears.

Using a suitable iteration of \eqref{eq:localdecay1} together with \eqref{eq:tnuAss2} and \eqref{eq:streamlineheuristic} we prove the following result.   

\begin{theorem} \label{thm:T2local}
    Suppose that $b\colon\T \to \R$ satisfies Assumption~\ref{Ass:classical} and let $\lambda_{\mathrm{min}} = \nu^{\frac{N+1}{N+3}}$, where $N \ge 1$ denotes the maximal order of vanishing of $b'$ at the critical points of $b$. Let $\bar{\lambda}_\nu(y)$ be given as in \eqref{eq:localrate} and define the modified local rate 
\begin{equation}\label{eq:localratelog}
       \Lambda_\nu(y) = \frac{\bar{\lambda}_\nu(y)}{1+|\log\nu|^{4N}} + \lambda_{\mathrm{min}}.
   \end{equation}
   There exist constants $C,c > 0$ and a function $R:(0,1] \to (0,\infty)$ with 
   $$ \lim_{\nu \to 0} \frac{R(\nu)}{\nu^p} = 0 \quad \forall p\ge 0 $$
   such that for all $\nu \in (0,1]$, mean-zero $f_0 \in L^\infty(\T^2)$, and $y \in \T$, the solution of \eqref{eq:ADEbshear} satisfies
    \begin{equation} \label{eq:localdecay}
   \left\|f(t,\cdot,y)\right\|_{L^\infty(\T)} \le C\left(e^{-c\Lambda_\nu(y)t}+R(\nu)e^{-c\lambda_{\mathrm{min}}t}\right)\|f_0\|_{L^\infty}
    \end{equation}   
    for all $t \ge 0$.
\end{theorem}

\begin{remark} \label{rem:error}
    The logarithmic loss in $\Lambda_\nu(y)$ and the error term of size $R(\nu)$ that decays with the minimal rate $\lambda_{\mathrm{min}}$ appear necessary. Together, they describe how mass of the measure $\Pt_t(x,y,\ccdot)$ can slowly diffusive away from the streamline at height $y$ into regions where the local decay rate is slightly worse.
\end{remark}

\begin{remark}
    The bound \eqref{eq:localdecay} results from one particular choice of how to balance the loss in the local rate (compared to $\bar{\lambda}_\nu$) with the error term $R(\nu)$. Other choices are possible, which result in different local decay estimates. For instance, if one is willing to introduce an arbitrarily small polynomial-in-$\nu$ loss in $\Lambda_\nu$, then the corresponding error term can be shown to satisfy $\lim_{\nu \to 0} \exp(\nu^{-p}) R(\nu) = 0$ for some $p  > 0$.
\end{remark}

\begin{remark} \label{rem:rough}
    While we restricted Theorem~\ref{thm:T2local} to the setting of Assumption~\ref{Ass:classical}, the proof of Theorem~\ref{thm:T2general} also gives local estimates that describe an improved rate of enhanced dissipation (compared to the $\nu^{1/3}$ rate for shears with $|b'(y)| \approx 1$) nearby points where $b'$ becomes singular. For example, if $b(\bar{y}+h) - b(\bar{y}) \propto |h|^\alpha$ for $\alpha \in (0,1)$ near some point $\bar{y}$, then it can be shown that there exists $c > 0$ such that for all $|y-\bar{y}| \ll 1$ and $0 < \nu \ll 1$ we have 
    $$ \frac{1}{t_\nu(y)} \ge c \lambda_\nu(y), \quad \text{where} \quad \lambda_\nu(y) = 
    \begin{cases} 
    \nu^{\frac{\alpha}{\alpha+2}} & \quad \text{if} \quad |y-\bar{y}| \le \nu^{\frac{1}{\alpha+2}}, \\ 
    \nu^{1/3}|y-\bar{y}|^{\frac{2(\alpha-1)}{3}} & \quad \text{if} \quad \nu^{\frac{1}{\alpha+2}} < |y-\bar{y}| \ll 1
    \end{cases} 
    $$
    and $t_\nu(y)$ is as defined in \eqref{eq:tnuintro}. Since $\nu^{\frac{\alpha}{\alpha+2}} \gg \nu^{1/3}$ for small $\nu$, $\lambda_\nu(y)$ quantifies how the local decay rate improves in the vicinity of $\bar{y}$ and statements analogous to Theorem~\ref{thm:T2local} can be shown. As mentioned just after \eqref{eq:introdecay}, that roughness can shorten the enhanced dissipation timescale has been observed previously in the literature. For example, the global decay estimate $\|f(t)\|_{L^2} \le C e^{-c\nu^\frac{\alpha}{\alpha+2}t}\|f_0\|_{L^2}$ has been proven for shears which are sharply $\alpha$-H\"{o}lder continuous at each point in space \cite{wei2021diffusion,CZCW20}.        
\end{remark}

\subsection{Results for radial shears on $\R^2$ or the unit disk}

Let $D$ denote either $\R^2$ or the unit disk $\mathbb{D} = \{\mathbf{x} \in \R^2: |\mathbf{x}| \le 1\}$. If $D = \R^2$, then we write $I = [0,\infty)$ and in the case of the unit disk we let $I = [0,1]$. Here we state our results for radially symmetric shear flows $u:D \to \R^2$ given in polar coordinates $(r,\theta) \in I \times [0,2\pi)$ by 
\begin{equation} \label{eq:R2shear}
    u(r,\theta) = r b(r) \hat{e}_\theta,
\end{equation}
where $b\colon I \to \R$. In this case, \eqref{eq:ADE} becomes 
\begin{equation} \label{eq:ADEpolarshear}
   \begin{cases}
    \partial_t f + b(r)\partial_\theta f = \nu \Delta f, \\ 
    f|_{t=0} = f_0,
    \end{cases}
\end{equation}
where 
$$ \Delta = \partial_{rr} + \frac{1}{r}\partial_r + \frac{1}{r^2}\partial_{\theta \theta} $$
is the polar Laplacian. In the case that $D$ is the disk, \eqref{eq:ADEpolarshear} is supplemented with no-flux boundary conditions. The streamline average 
\begin{equation}
    \langle f \rangle(t,r):=\frac{1}{2\pi}\int_0^{2\pi} f(t,r,\theta)\dee \theta
\end{equation}
simply solves
\begin{equation} \label{eq:radialheat}
        \partial_t \langle f \rangle(t,r) = \nu \Big( \partial_{rr} + \frac{1}{r}\partial_r \Big)\langle f \rangle(t,r)
\end{equation}
and hence decays to zero at best only on the slow $\nu^{-1}$ diffusive timescale. Therefore, just as in \eqref{eq:T2meanfree}, we will always assume without loss of generality that 
\begin{equation} \label{eq:R2meanfree}
\int_0^{2\pi}f_0(r, \theta)\dee \theta = 0 \quad \forall r \in I, 
\end{equation}
which is preserved by the evolution of \eqref{eq:radialheat}.

We will consider a general setting where $b$ is smooth and possesses finitely many critical points. 
The precise conditions we impose are stated as follows.

    \begin{assumption}\label{ass:R2_general}
        The function $b\colon I\setminus\{0\} \to \mathbb{R}$ is smooth, possesses at most finitely many critical points $\{r_1,\ldots, r_M\} \subseteq \mathrm{int}(I)$, and there exists $N \in \N$ such that for each $1 \le i \le M$ we have 
        $$ b^{(k)}(r_i) \neq 0 \quad \text{for some} \quad k \le N+1.$$
        Moreover:
        
        \begin{itemize}
            \item[(a)] The derivative $b'\colon I \to \R$ is continuous and there exists $q \ge 0$ such that 
            \begin{equation} \label{eq:qdef}
            \lim_{r \to 0}\frac{b'(r)}{r^{q}} \in \R\setminus \{0\}. 
            \end{equation}
            \item[(b)]  If $D = \R^2$, then there exist constants $c,C > 0$ such that 
        \begin{equation} \label{eq:R2assumption}
        \liminf_{r \to \infty}|b'(r)| \ge c \quad \text{and} \quad \limsup_{r \to \infty}\left|\frac{b''(r)}{b'(r)}\right| \le C.
        \end{equation}
        \end{itemize}

    \end{assumption}

    \begin{remark}
    In the case that $D$ is the unit disk, Assumption~\ref{ass:R2_general} is satisfied by any function
    $b\colon [0,1] \to \R$ that is smooth, possesses only finitely many critical points with $b'$ vanishing to finite order at each of them, and satisfies $\lim_{r \to 1}|b'(r)| \neq 0$. The assumption that the $b'$ does not vanish at the boundary is not essential and made only for convenience. Condition (b) says essentially that if $D = \R^2$, then $b$ grows at least linearly but no faster than exponentially as $r \to \infty$. Typical examples of functions satisfying both (a) and (b) above then include $e^r - 1$, $r^p$, and $r^p e^r$, where $p \ge 1$ is an arbitrary exponent.
    \end{remark}

    \begin{remark}
        It is possible to allow for behavior near $r = 0$ that is more general than what we state in (a), but we do not do this for the sake of simplicity.  
    \end{remark}

    Just like Assumption~\ref{Ass:T2general} made earlier,
    Assumption~\ref{ass:R2_general} implies that there exists $0 < h_0
    \ll 1$ such that for every $r \in I$ there is a choice $\iota(r)
    \in \{1,-1\}$ so that no elements of $\{0,r_1,\ldots,r_M\}$ lie
    strictly between $r$ and $r+\iota(r)h_0$ and $h \mapsto
    |b(r+\iota(r)h) - b(r)|$ is monotone increasing for $h \in
    [0,h_0]$. Just as in the $\T^2$ setting, it is natural then to
    consider the local timescale $t_\nu\colon I \to (0,\infty)$ defined by 
    \begin{equation}\label{eq:tnurintro}
        t_\nu(r) = \inf\{t \ge 0: t|b(r+\iota(r)\sqrt{\nu t}) - b(r)| = 1\}.
    \end{equation}
    Defining $N$ to be the maximal order of vanishing of $b'$ at the critical points $\{r_1,\ldots,r_M\}$ and $n = \max(q,N)$, it is straightforward to use Taylor's theorem and Assumption~\ref{ass:R2_general} to show that there is a constant $C > 0$ such that
    \begin{equation} \label{eq:tnurbound}
    \sup_{r \in I} t_\nu(r) \le C \nu^{-\frac{n+1}{n+3}}.
    \end{equation}
    Our main result in the setting of Assumption~\ref{ass:R2_general} is then stated as follows.
    
    \begin{theorem} \label{thm:RadialGen}
    Let $b\colon I \to \R$ satisfy Assumption~\ref{ass:R2_general},
    $t_\nu\colon  I \to (0,\infty)$ be as defined in
    \eqref{eq:tnurintro}, and $f\colon  [0,\infty) \times I \times [0,2\pi) \to \R$ denote the solution of \eqref{eq:ADEpolarshear}, which is supplemented with no-flux boundary conditions in the case that $D$ is the unit disk. Then, there exist constants $\epsilon, C_0$ such that for any $r > 0$ and $f_0 \in L^\infty(I\times[0,2\pi))$ satisfying \eqref{eq:R2meanfree} we have
    \begin{equation} \label{eq:localdecay2}
\|f(C_0t_\nu(r),r,\ccdot)\|_{L^\infty([0,2\pi))} \le (1-\epsilon)\|f_0\|_{L^\infty}.
    \end{equation}
    In particular, by iterating \eqref{eq:localdecay2} and using \eqref{eq:tnurbound}, there exist $c,C > 0$ such that
\begin{equation} \label{eq:radialglobal}
   \|f(t)\|_{L^\infty} \le C e^{-c \nu^\frac{n+1}{n+3}t}\|f_0\|_{L^\infty}
\end{equation}
for all $t \ge 0$.
    \end{theorem}

    \begin{remark}
    In a fashion quite similar to \eqref{eq:localrate} and
    \eqref{eq:tnuAss2}, the local rate $\lambda_\nu(r) = 1/t_\nu(r)$
    associated with $t_\nu \colon I \to (0,\infty)$ can be easily estimated. As in Theorem~\ref{thm:T2local}, it is then possible with a suitable iteration procedure of the bound on measures that gives \eqref{eq:localdecay2} to upgrade to a global-in-time estimate that says $\|f(t,r,\ccdot)\|_{L^\infty([0,2\pi))}$ decays, up to small corrections, at the rate $\lambda_\nu(r)$. 
    \end{remark}

    \begin{remark} \label{rem:vortex}
        The regularity of $b$ assumed in \eqref{eq:qdef} does not play a fundamental role in the proof of Theorem~\ref{thm:RadialGen}. By slightly modifying the proof and treating places where $b'$ blows up in the same way that we do in Theorem~\ref{thm:T2general}, we can also obtain an enhanced dissipation result for velocity fields corresponding to singular vortices. Specifically, if $b\colon [0,1] \to \R$ is given by $b(r) = r^{-\alpha}$ for some $0 < \alpha \le 2$ then we can show that the solution of \eqref{eq:ADEpolarshear} supplemented with no-flux boundary conditions satisfies
 \begin{equation}\label{eq:Singdecay}
     \|f(t)\|_{L^\infty} \le C e^{-c\nu^{1/3}t}\|f_0\|_{L^\infty}
 \end{equation}
 for some constants $c,C > 0$.
\end{remark}

\section{Reduction to a quantitative control problem} \label{sec:probability}

The goal of this section is to carry out the necessary probabilistic arguments to reduce the proofs of our main results to solving suitable stochastic control problems. 

We begin with some notation and terminology. For concreteness, we
suppose that we are in the $\T^2$ setting, but the same notations
extend in the expected way for radially symmetric shears. For purely
technical reasons, it will be convenient to consider
\eqref{eq:SDEintro} first as an SDE on $\R^2$ and then map the
solution back onto $\T^2$. Thus, given a continuous function $b\colon \T \to \R$, we extend it to a periodic function defined on $\R$ and let $(x_t,y_t) \in \R^2$ denote the solution of the SDE
\begin{equation} \label{eq:SDE2}
\begin{cases}
\dee x_t = -b(y_t)\dt + \sqrt{2\nu} \dee W_t, \\ 
\dee y_t = \sqrt{2\nu} \dee B_t,
\end{cases}
\end{equation}
where $W_t$ and $B_t$ are independent Brownian motions on $\R$ defined on the filtered probability space $(\Omega, \mathcal{F}, \mathcal{F}_t, \P)$. Here, $\{\mathcal{F}_t\}_{t \ge 0}$ is the natural filtration generated by $(W_t,B_t)$. For any initial condition $(x_t,y_t)|_{t= 0 } = (x,y) \in \R^2$, the SDE \eqref{eq:SDEintro} has a unique global strong solution $(x_t,y_t)_{t\ge 0}$ and hence we may introduce the corresponding flow map $\xi:\R_+ \times \R^2 \times \Omega \to \R^2$ defined by $\xi(t,x,y,\omega) = (x_t(\omega),y_t(\omega))$. 
The solution $(x_t,y_t)$ gives rise to a Markov process on $\R^2$ with Markov transition kernel $\Qt_t$ defined by
\begin{equation} \label{eq:kerneldef}
    \Qt_t(x,y,A) = \P(\xi(t,x,y,\ccdot) \in A) \qquad (x,y) \in \T^2, \quad A \in \mathcal{B}(\R^2),
\end{equation}
where $A \in \mathcal{B}(\R^2)$ denotes a Borel subset of $\R^2$. Recall that a Markov transition kernel on a measurable space $(Z,\mathcal{A})$ is a mapping $Q:Z \times \mathcal{A} \to [0,1]$ such that $z \mapsto Q(z,A)$ is measurable for fixed $A \in \mathcal{A}$ and $Q(z,\ccdot)$ is a probability measure for every $z \in Z$. The Markov semigroup $\{\Qt_t\}_{t \ge 0}$ associated with \eqref{eq:kerneldef}, which acts on the space $B_b(\R^2)$ of bounded and Borel measurable functions on $\R^2$, is defined by 
\begin{equation} \label{eq:semigroupdef}
    \Qt_t g(x,y) = \int_{\R^2}g(x',y')\Qt_t(x,y,\dee x',\dee y').
\end{equation}
We also have the equivalent formula 
\begin{equation}
    \Qt_t g(x,y) = \E g(\xi(t,x,y,\ccdot):= \int_\Omega g(\xi(t,x,y,\omega))\P(\dee \omega) \quad \forall g \in B_b(\R^2).
\end{equation}
It will be useful at times to consider the dual semigroup on measures, which acts on a Borel probability measure $\mu$ to produce a new probability measure $\mu \Qt_t$ defined by
\begin{equation} \label{eq:dualsemigroup}
    \mu \Qt_t(A) = \int_{\R^2} \Qt_t(x,y,A) \mu(\dx,\dy) \quad \forall A \in \mathcal{B}(\R^2).
\end{equation}
By mapping the solution $(x_t,y_t)$ of \eqref{eq:SDE2} back to $\T^2$, the transition kernel $\Qt_t$ induces a new Markov transition kernel $\Pt_t$ on $\T^2 \cong [0,1)^2$ defined by
\begin{equation} \label{eq:periodizedtransition}
    \Pt_t(x,y,A) = \P(\pi(\xi(t,x,y,\ccdot)) \in A) \quad (x,y) \in \T^2, \quad A \in \mathcal{B}(\T^2),
\end{equation}
where $\pi: \R^2 \to \T^2$ denotes the natural projection defined by $\pi(x,y) = (x\mod 1, y\mod 1)$. The Markov semigroup and dual action on measures associated with $\Pt_t$ (now acting on functions and measures on $\T^2$) are defined in the same way as before through \eqref{eq:semigroupdef} and \eqref{eq:dualsemigroup}, respectively. The connection between \eqref{eq:SDE2} and its corresponding PDE is then that for any $f_0 \in B_b(\T^2)$, the unique solution of \eqref{eq:ADEbshear} is given by \begin{equation} \label{eq:probrepT}
f(t,x,y) = \Pt_t f_0(x,y) = \E f_0(\pi(\xi(t,x,y,\ccdot))).
\end{equation}

Phrased in the $\T^2$ setting, the main result of this section reduces the proof of Theorem~\ref{thm:T2general} and \eqref{eq:localdecay1} to showing that one can cause two solutions of \eqref{eq:SDE2} starting from distinct initial conditions on the same streamline to reach each other in finite time (with some positive probability) by perturbing $B_t$ in one of the equations by a suitable control function. The precise statement is as follows.

\begin{proposition} \label{prop:T2control}
    Let $\xi: \R_+ \times \R^2 \times \Omega \to \T^2$ be as defined above and for a function $V_t:[0,\infty) \times \Omega \to \R$ adapted to the filtration $\mathcal{F}_t$, let $\tilde{\xi}:\R_+ \times \R^2 \times \Omega \to \R^2$ denote the solution map associated with the controlled SDE 
    \begin{equation} \label{eq:ControlledSDE}
    \begin{cases}
        \dee \tilde{x}_t = -b(\tilde{y}_t) \dt + \sqrt{2\nu} \dee W_t, \\ 
        \dee \tilde{y}_t = \sqrt{2\nu}(\dee B_t + V_t \dt).
        \end{cases}
    \end{equation}
    Suppose that there are constants $c_0,C_0 > 0$ such that all $x,\tilde{x},y \in \T$ there exist $T_y < \infty$ and $V_t:[0,\infty) \times \Omega \to \R$ such that 
    \begin{equation} \label{eq:controlprop}
        \P \Big(\int_0^\infty |V_t|^2 \dt \le C_0\Big)=1 \quad \text{and} \quad \P(\xi(T_y,x,y,\omega) = \tilde{\xi}(T_y,\tilde{x},y,\omega)) \ge c_0.
    \end{equation}
    Then, there exists $C > 0$ such that for any $K \subseteq \T$ and probability measures $\mu_1,\mu_2$ that have the same $y$-marginal distributions and satisfy $\mu_i(\T \times K) = 1$, we have
\begin{equation} \label{eq:localdecayprop}
\|\mu_1\Pt_{T_*} - \mu_2 \Pt_{T_*}\|_{TV} \le (1-\epsilon)\|\mu_1 - \mu_2\|_{TV}, \quad \text{where} \quad T_* = \sup_{y \in K} T_y.
\end{equation}
In particular, assuming that $T:= \sup_{y \in \T} T_y < \infty$ an
d setting $\lambda = 1/T$, there exist $c,C > 0$ such that 
\begin{equation}\label{eq:globaldecayprop}
        \|\Pt_t f\|_{L^\infty} \le C e^{-c\lambda t}\|f\|_{L^\infty}
    \end{equation}
    for all $t \ge 0$ and $f \in L^\infty(\T^2)$ satisfying \eqref{eq:T2meanfree}.
\end{proposition}

The proof of Proposition~\ref{prop:T2control} will be carried out in two main steps. First, in Section~\ref{sec:marginal}, we show that \eqref{eq:localdecayprop} reduces to proving that 
\begin{equation} \label{eq:TVprop}
\|\delta_{(x,y)}\Pt_{T_y} - \delta_{(\tilde{x},y)}\Pt_{T_y}\|_{TV} \le (1-\epsilon),
\end{equation}
where $\|\cdot\|_{TV}$ denotes the total variation distance on probability measures. Next, in Section~\ref{sec:girsanov}, we use a suitable application of Girsanov's theorem to prove that \eqref{eq:TVprop} is implied by the control assumption \eqref{eq:controlprop}. The results of Sections~\ref{sec:marginal} and~\ref{sec:girsanov} will be carried out in a setting general enough to also immediately imply a version of Proposition~\ref{prop:T2control} in the radially symmetric case (see Proposition~\ref{prop:R2control} stated at the end of this section).

\subsection{Convergence on marginals} \label{sec:marginal}

In this section, $P_t$ denotes a Markov transition semigroup on $X\times Y$,
where $X, Y$ are two Polish spaces, and we use the same notation $\mu P_t$ defined in \eqref{eq:dualsemigroup} for the dual action of $P_t$ on Borel probability measures on $X\times Y$. For $(x,y) \in
X\times Y$ define the coordinate projections $\Pi_x (x,y)=x$ and
$\Pi_y(x,y)=y$. By the disintegration theorem, any probability measure $\mu$ on $X\times Y$ can be factored into a marginal probability measure $\bar \mu= \mu \Pi_y^{-1}$ on $Y$
and a collection of conditional measures $\mu^y$ on $X$ for $\bar
\mu$-almost every $y \in Y$ so that $\mu(\dx \times \dy ) = \mu^y(\dx)
\bar \mu(\dy)$ (see e.g. \cite{Chang_Pollard_1997,Simmons_2012}). Here, $\mu \Pi_y^{-1}(A) := \mu(\Pi_y^{-1}(A))$ for a measurable set $A \subseteq Y$.

\begin{proposition}\label{thm:marginal measures}
  Let $\mu_1$ and $\mu_2$ be two probability measures on $X\times Y$
  such that $\mu_1 \Pi_y^{-1}=\mu_2 \Pi_y^{-1}$ and $\mu_i \Pi_y^{-1}(K)=1$ for some $K
  \subseteq Y$. If $T\colon K \rightarrow (0,\infty)$ is a bounded
  function so that
  \begin{align}\label{eq:twoPoints}
    \| \delta_{(x,y)} P_{T(y)} - \delta_{(\tilde{x},y)} P_{T(y)}\|_{TV} \leq 1-\epsilon 
  \end{align}
  for some $\epsilon \in (0,1)$ and all $x,\tilde{x} \in X$ and $y \in Y$,  then
  \begin{align*}
      \|\mu_1 P_{T_*} - \mu_2 P_{T_*}\|_{TV} \leq (1-\epsilon )\|\mu_1
    - \mu_2\|_{TV},
  \end{align*}
  where $T_* = \sup_{y \in K} T(y)$. In particular, by the Hahn-Jordan decomposition, for any signed measure $\mu$ on $X \times Y$ with $\mu(X \times Y) = 0$, $\mu \Pi_y^{-1} = 0$, and $\|\mu\|_{TV(X\times K^c)} = 0$ there holds
  $$ \|\mu \Pt_{T_*}\|_{TV} \le (1-\epsilon)\|\mu\|_{TV}. $$
\end{proposition}

Before we prove Proposition~\ref{thm:marginal measures}, we recall some
basic facts about the total-variation distance between measures and
couplings of measures. Given two probability measures $\mu_1$ and $\mu_2$ on
a Polish space $X$, the set of couplings
$\mathcal{C}(\mu_1,\mu_2)$ is defined by
\begin{align*}
  \mathcal{C}(\mu_1,\mu_2) =\Big\{ \Gamma \in \mathcal{M}_1(
  X\times X) : \Gamma \Pi_i^{-1} =\mu_i\Big\},
\end{align*}
where $\mathcal{M}_1(X)$ is the space of probability measures on $X$ and
$\Pi_i$ is the coordinate projection defined by $\Pi_i(x_1,x_2)=x_i$.
Then, given any $\mu_1, \mu_2 \in \mathcal{M}_1(X)$, we recall that the
total variation distance has the two complementary definitions
\begin{equation} \label{eq:tvdef}
  \|\mu_1-\mu_2\|_{TV} = \sup_{|\phi|_\infty \leq 1} \int_X
  \phi \dee\mu_1 -  \int_X
  \phi \dee\mu_2= \inf_{\Gamma \in \mathcal{C}(\mu_1,\mu_2)} \Gamma \big(
  \{(x_1,x_2) \in X\times X: x_1\neq x_2\}\big),
\end{equation}
where the supremum is taken over bounded, measurable functions $\phi:X \to \R$ and $|\phi|_\infty = \max_{x \in X}|\phi(x)|$. While the coupling $\Gamma$ that realize the minimum is not unique,
there is a canonical choice often referred to as the \textit{maximal
coupling}. We can always find a measure $\mu$ so that $\mu_1$ and $\mu_2$
are both absolutely continuous with respect to $\mu$ with
$\mu_i(\dx)=\rho_i(x)\mu(\dx)$. Then, the maximal coupling is given by
\begin{align*}
  \Gamma_*(\dee x_1\times \dee x_2) & = (\rho_1 \wedge
  \rho_2)(x_1)\mu(\dee x_1)\delta_{x_1}(\dee x_2)\\
  & \quad \quad \quad  + \frac{(\rho_1 - \rho_1 \wedge
  \rho_2)(x_1)\mu(\dee x_1) \times  (\rho_2 - \rho_1 \wedge
  \rho_2)(x_2)\mu(\dee x_2)}{1-\int_X (\rho_1 \wedge \rho_2)(x)\mu(\dx)},
\end{align*}
where $a \wedge b = \min(a,b)$. It is not hard
to see that the measures $ (\rho_1 - \rho_1 \wedge
  \rho_2)(x) \mu(\dx)$ and $(\rho_2 - \rho_1 \wedge
  \rho_2)(x) \mu(\dx)$ are mutually singular, so $\|\mu_1
  -\mu_2\|_{TV} = \Gamma_*\big(
  \{(x_1,x_2) \in X\times X: x_1\neq x_2\}\big) =1-  \int_X (\rho_1 \wedge
  \rho_2)(x) \mu(\dx)$.

\begin{lemma}\label{prop:fiber_TV}
  Let $\mu_1$ and $\mu_2$ be probability measures on $X\times Y$
  with $\mu_1 \Pi_y^{-1}=\mu_2 \Pi_y^{-1}$. Then,
  \begin{align*}
    \|\mu_1 -\mu_2\|_{TV} &=  \int_Y \|\mu_1^y- \mu_2^y\|_{TV} \bar
                            \mu(\dy)\\
    & =\inf_{\Gamma \in \mathcal{C}_y(\mu_1,\mu_2) } \Gamma \big(\{
  (x_1,y_1,x_2,y_2) \in (X\times Y)^2: (x_1,y_1)\neq (x_2,y_2)\}\big),
  \end{align*}
  where $\mu_1^y$, $\mu_2^y$, and $\bar \mu$ are as in the
  disintegration $\mu_i(\dx \times \dy) = \mu^y(\dx) \bar \mu(\dy)$
  discussed at the start of the section and
  \begin{multline*}
    \mathcal{C}_y(\mu_1,\mu_2)= \Big\{\Gamma \in
    \mathcal{C}(\mu_1,\mu_2) : \Gamma(\dee x_1 \times \dee y_1 \times
    \dee x_2\times \dee y_2)\\=  \Gamma^{y_1}(\dee x_1\times \dee x_2) \bar \mu(\dee y_1) \delta_{y_1}(\dee y_2)
    \text{ with }  \Gamma^y \in \mathcal{C}(\mu^y_1, \mu^y_2) \Big\}
  \end{multline*}
\end{lemma}

\begin{remark}
While the inequality $\|\mu_1 - \mu_2\|_{TV} \le \int_Y \|\mu_1^y - \mu_2^y\|_{TV} \bar{\mu}(\dy)$ is true for other metrics on measures defined through duality (as in the first equality of \eqref{eq:tvdef}, but with $|\cdot|_\infty$ replaced by a different norm), for instance the Wasserstein metric, having equality is unique to the total variation distance. 
\end{remark}

\begin{proof}
  For any bounded and measurable test function $\phi: X\times Y \rightarrow \R$, we have
    \begin{align*}
            \mu_1\phi - \mu_2\phi &= \int_{Y} \int_X \phi(x, y) [\mu_1^y - \mu_2^y](\dx)\overline{\mu}(\dy) \\
            &= \int_Y [\mu_1^y\phi_y - \mu_2^y\phi_y]\,\overline{\mu}(\dy),
    \end{align*}
  where $\phi_y(x) = \phi(x,y)$ is viewed as a function from $X$ into
  $\R$ for fixed $y \in Y$. Thus
  \begin{align*}
    \|\mu_1 -\mu_2\|_{TV}=\sup_{|\phi|_\infty \leq 1}  \int_Y  [\mu_1^y\phi_y - \mu_2^y\phi_y]\,\overline{\mu}(\dy)
    &\leq  \int_Y   \sup_{|\phi_y|_\infty \leq 1} [\mu_1^y\phi_y - \mu_2^y\phi_y]\,\overline{\mu}(\dy)\\
    &=  \int_Y \|\mu_1^y- \mu_2^y\|_{TV} \bar \mu(\dy).
  \end{align*}
  All that remains is to show that there exists $\Gamma \in
  \mathcal{C}_y(\mu_1,\mu_2)$ with $\|\mu_1-\mu_2\|_{TV}=\Gamma\big(
  \{(x_1,x_2) \in X\times X: x_1\neq x_2\}\big)= \int_Y \|\mu_1^y- \mu_2^y\|_{TV} \bar
  \mu(\dy)$. To this end, observe that if $\mu_i(\dx \times \dy) =
  \rho_i(x,y) \nu(\dx \times \dy)$ and $\nu(\dx \times \dy)=\nu^y(\dx)\bar
  \nu(\dy)$, then $\bar{\mu}$ is absolutely continuous with respect so $\bar{\nu}$ so that $\bar{\mu}(\dy) = \bar{\rho}(y) \bar{\nu}(\dy)$ for some function $\bar{\rho}$. We may then write $\rho_i(x,y) =
  \rho_i^y(x)\bar \rho(y)$, so that
  $\mu^y_i(\dx)= \rho^y_i(x) \nu^y(\dx)$. Thus, if we define
  \begin{align*}
    \Gamma^y_*(\dee x_1\times \dee x_2) &= (\rho_1^y \wedge
    \rho_2^y)(x_1)\nu^y(\dee x_1) \delta_{x_1}(\dee x_2) \\
    & \quad \quad + \frac{(\rho_1^y-
    \rho_1^y \wedge \rho_2^y)(x_1)\nu^y(\dee x_1) \times
    (\rho_2^y- \rho_1^y \wedge \rho_2^y)(x_2)\nu^y(\dee x_2)}{1-\int_X (\rho_1^y \wedge \rho_2^y)(x) \nu^y(\dx)},
  \end{align*}
  then $\Gamma^y_* \in \mathcal{C}(\mu_1^y,\mu_2^y)$ with 
  \begin{equation} \label{eq:couple1}
  \Gamma^y_*( \{(x_1,x_2) : x_1 \neq x_2\}) = \| \mu_1^y - \mu_2^y\|_{TV} = 1- \int_X (\rho_1^y \wedge \rho_2^y)(x)\nu^y(\dx).
  \end{equation} 
  We may then define $\Gamma_{**} \in
  \mathcal{C}_y(\mu_1,\mu_2)$ by $\Gamma_{**}(\dee x_1\times \dee y_1\times \dee x_2 \times \dee y_2)= \Gamma^{y_1}_*(dx_1
    \times \dee x_2)\bar \mu(\dee y_1) \delta_{y_1}(\dee y_2)$. Recalling that $\nu^y(\dx) \bar{\mu}(\dy) = \bar{\rho}(y)\nu(\dx \times \dy)$ and employing \eqref{eq:couple1} we have
    \begin{equation} \label{eq:couple2}
    \begin{aligned}
\Gamma_{**}(\{(x_1,y_1,x_2,y_2) : (x_1,y_1)\neq (x_2,y_2)\})&=\int_Y \|\mu_1^y- \mu_2^y\|_{TV} \bar
 \mu(dy) \\ 
 & = 1 - \int_{X \times Y} (\rho_1^y \wedge \rho_2^y)(x)\bar{\rho}(y) \nu(\dx \times \dy).
    \end{aligned}
\end{equation}
Since 
$$ (\rho_1^y \wedge \rho_2^y)(x) \bar{\rho}(y) = (\rho_1^y(x) \bar{\rho}(y)) \wedge (\rho_2^y(x)\bar{\rho}(y)) = (\rho_1 \wedge \rho_2)(x,y), $$
we see then from \eqref{eq:couple2} that 
$$ \Gamma_{**}(\{(x_1,y_1,x_2,y_2) : (x_1,y_1)\neq (x_2,y_2)\}) = 1- \int_{X\times Y} (\rho_1 \wedge \rho_2)(x,y)\nu(\dx \times \dy) = \|\mu_1 - \mu_2\|_{TV},$$
where the second equality follows from the characterization of the total variation distance discussed earlier. This completes the proof.
\end{proof}

With Lemma~\ref{prop:fiber_TV} in hand, we return to the proof
of Proposition~\ref{thm:marginal measures}.
\begin{proof}[Proof of Proposition~\ref{thm:marginal measures}]
     Let $\Gamma^y_* \in \mathcal{C}(\mu_1^y, \mu_2^y)$ be a measure
     that realizes the distance $||\mu_1^y - \mu_2^y||_{TV}$, namely
     \begin{align*}
       ||\mu_1^y - \mu_2^y||_{TV} &= \int_{X\times X} \one_{x\neq \tilde{x}}\,\Gamma^y_*(\dx, \dee \tilde{x})
 = \int_{X\times X} || \delta_{(x, y)} - \delta_{(\tilde{x}, y)} ||_{TV}\,{\Gamma}^y_*(\dee x, \dee \tilde{x}),
     \end{align*}
     where we have used the fact that $ \one_{x\neq x'} =|| \delta_{(x, y)} - \delta_{(x', y)} ||_{TV}$. Now, for $x,\tilde{x} \in X$ and $y \in Y$, let
     $\Gamma^y_{x, \tilde{x}} \in \mathcal{C}(\delta_{(x,
       y)}P_{T_\ast}, \delta_{(\tilde{x}, y)}P_{T_\ast})$
     be a measure that realizes
     $|| \delta_{(x, y)}P_{T_\ast} - \delta_{(\tilde{x},
       y)}P_{T_\ast}) ||_{TV}$ and define
\begin{align*}
  m(\dee x_1\times \dee y_1, \dee x_2\times \dee y_2) = \int_Y\int_{X\times X} \Gamma^y_{x, \tilde{x}}(\dee x_1\times \dee y_1, \dee x_2\times \dee y_2)\Gamma_*^y(\dee x, \dee \tilde{x})\overline{\mu}(\dy). 
\end{align*}
One can easily check that
$m \in \mathcal{C}(\mu_1P_{T_\ast},
\mu_2P_{T_\ast})$.  Since the total variation distance is the
minimum over couplings, we then have that
\begin{align*}
  ||\mu_1P_{T_\ast} - \mu_2P_{T_\ast}||_{TV} &\leq  m\Big(\{(x_1,y_1,x_2,y_2)\in (X\times Y)^2 : (x_1,y_1) \neq (x_2,y_2)\}\Big)\\            &=\int_{K}\int_{X\times X}\Gamma^y_{x, \tilde{x}}\Big(\{(x_1,y_1,x_2,y_2)\in (X\times Y)^2 : (x_1,y_1) \neq (x_2,y_2)\}\Big)\Gamma_*^y(\dee x, \dee \tilde{x})\overline{\mu}(\dy) \\                    &= \int_{K}\int_{X\times X} || \delta_{(x, y)}P_{T_\ast} - \delta_{(\tilde{x}, y)}P_{T_\ast} ||_{TV} \Gamma_*^y(\dee x, \dee \tilde{x})\overline{\mu}(\dy). 
\end{align*}
Since the assumption in \eqref{eq:twoPoints} can be restated as $$\|
\delta_{(x,y)} P_{T(y)} - \delta_{(\tilde{x},y)} P_{T(y)}\|_{TV} \leq (1-\epsilon)  || \delta_{(x, y)} - \delta_{(\tilde{x}, y)} ||_{TV}$$
and $\|\delta_{(x,y)}P_{T_*} - \delta_{(\tilde{x},y)}P_{T_*}\|_{TV} \le \|
\delta_{(x,y)} P_{T(y)} - \delta_{(\tilde{x},y)} P_{T(y)}\|_{TV}$ for any $y \in K$, continuing
the above reasoning and applying Lemma~\ref{prop:fiber_TV} produces
\begin{align*}
  ||\mu_1P_{T_\ast} - \mu_2P_{T_\ast}||_{TV} &\leq (1 - \epsilon)\int_{K}\int_{X
  \times X} || \delta_{(x, y)} - \delta_{(\tilde{x}, y)} ||_{TV} \Gamma_*^y(\dee x, \dee \tilde{x})\overline{\mu}(\dy) \\
&\leq (1 - \epsilon)\int_{K} || \mu_1^y - \mu_2^y||_{TV} \overline{\mu}(\dy) = (1 - \epsilon) ||\mu_1 - \mu_2||_{TV},
\end{align*}
which completes the proof.
\end{proof}

\subsection{Controlling the total variation distance with Girsanov's theorem } \label{sec:girsanov}

Fix $d,m \ge 1$ and let $f\colon \R^d \rightarrow \R^d$ and $g\colon \R^d \rightarrow \R^m
\times \R^d$ be measurable functions such that $X_t$ defined by the It\^{o} stochastic differential
equation
\begin{equation}
  \label{eq:abstract_sde}
  \begin{cases}
    \dee X_t = f(X_t)\dt + g(X_t)\dee W_t,\\
    X_0 =x
  \end{cases}
\end{equation}
has a unique strong solution for every $x \in \R^d$. Here $W_t=(W_t^{(1)}, \dots, W_t^{(m)})$ is an
$m$-dimensional Brownian motion defined on the Wiener space $(C([0,\infty); \R^m), \mathcal{B}(C([0,\infty); \R^m)), \P)$, where $\P$ denotes the standard $m$-dimensional Wiener measure. We write $X_t(x,W)$ for the
solution of \eqref{eq:abstract_sde} with initial condition $x$ and
Wiener process $W_t$. Let $P_t$ be the associated Markov transition
kernel generated by $X_t$. That is, $\delta_x P_t(A) = \P( X_t(x,W) \in A )$
for $x \in \R^d$ and $A \subseteq \R^d$.

Fixing a $T>0$ and $\widetilde{x} \in \R^d$,
 let $u\colon
  [0,T] \times  C([0,T]; \R^m) \times  \rightarrow \R^m$ be a measurable function such that $t \mapsto u(t,W)$
  is adapted to the filtration generated by $W$ and the It\^o equation
  \begin{equation}
  \label{eq:abstract_sde_control}
  \begin{cases}
    \dee \widetilde{X}_t &= f(\widetilde{X}_t)\dt + g(\widetilde{X}_t)\dee W_t+ g(\widetilde{X}_t)u(t,W) \dt,\\
    \widetilde{X}_0&=\widetilde{x}
  \end{cases}
\end{equation}
has well defined strong solution starting from $\widetilde{x}$. We will write
$\widetilde{X}_t(\widetilde{x},W)$ for the solution of
\eqref{eq:abstract_sde_control} starting from $\widetilde{x}$.
\begin{proposition}\label{thm:Girsanov}
In the above setting, assume that there exists a positive, finite
constant $C$ and time $T>0$ so that
$\P(\int_0^T |u(t,W)|^2 \dt \leq C) =1$. Then, if there exists 
  $\mathcal{A} \subset C([0,T]; \R^m)$ such that $\P( W \in
  \mathcal{A})\geq \alpha >0$ and  $X_T(x,W) =
 \widetilde{X}_T(\widetilde{x}, W)$ for all $W \in \mathcal{A}$, one has that
  \begin{align*}
    \| \delta_x P_T - \delta_{\widetilde{x}} P_T \|_{TV} \leq
    1-\tfrac14 \alpha^2 e^{-2C}.
  \end{align*}
\end{proposition}

The proof of this proposition relies on the following lemma, which
is adapted from those in \cite{Mattingly2002,Mattingly2003}.
\begin{lemma}\label{prop:TV-two-measures}
  Let $\mu_1$ and $\mu_2$ be two probability measures on a measure
  space $\mathbf{X}$. Suppose that $\mu_i \geq 
  \lambda_i$ for two measures $\lambda_i(\dx)= \rho_i(x) \lambda(\dx)$ such
  that $\lambda_1 \sim \lambda_2$.\footnote{Here,  $\mu_i \geq 
  \lambda_i$ means that  $\mu_i(A) \geq 
  \lambda_i(A)$ for all measurable sets $A$. Similarly  $\lambda_1
  \sim \lambda_2$, means that the two measures are
  equivalent in the sense that  $\lambda_1(A)>0$ if and only if
  $\lambda_2(A)>0$ for measurable $A$.} Define
$M(x)=\frac{\rho_2(x)}{\rho_1(x)}$. If for some $K,p>1$ there holds
\begin{align} \label{eq:Mbound}
     \int M(x)^{p+1} \lambda_1(dx) =\int M(x)^p \lambda_2(dx) \leq K,
\end{align}
then  \begin{align*}
      \|\mu_1 - \mu_2\|_{TV}   &\leq 1-\Big(1-\frac1p\Big)\Big(\frac{\lambda_1(\mathbf{X})^p}{pK}\Big)^{\frac1{p-1}}<1.
  \end{align*}
\end{lemma}

\begin{proof}
Because $\lambda_1$ and $\lambda_2$ are equivalent, we know
that $\lambda(\{x\in \mathbf{X} : M(x)=0\})=0$. Let $[a]^+ = \max(a,0)$ for $a \in \R$ and observe that since $M=M\wedge 1 +[M-1]^+$ we have
\begin{align} \label{eq:Meq}
 (\rho_1\wedge \rho_2)(x) \lambda(\dx)&=
  (M\wedge 1)(x) \rho_1(x) \lambda(\dx)= \rho_2(x) \lambda(\dx)-  \frac{[M-1]^+}{M}(x) \rho_2(x) \lambda(\dx).
\end{align}
Now, for any $L>0$, we define $B_L=\{x \in \mathbf{X} : M(x) \leq L\}$ and notice
that Markov's inequality and \eqref{eq:Mbound} imply that $\lambda_2(B_L^c) \leq \frac{K}{
  L^p}$. Since by \eqref{eq:Meq} we have
\begin{align*}
 \int   (\rho_1 \wedge \rho_2)\lambda(\dx)  &\geq \lambda_2(\mathbf{X})- \int_{B_L^c}\frac{[M-1]^+}{M}(x) \rho_2(x)
\lambda(\dx) -  \int_{ B_L}\frac{[M-1]^+}{M}(x) \rho_2(x)
\lambda(\dx)
\end{align*} 
and moreover $0 \leq \frac{[M-1]^+}{M}(x) \leq \frac{L-1}{L}$ for all $x
\in B_L$ and $\frac{L-1}{L} \leq \frac{[M-1]^+}{M}(x)\leq 1$ for  $x
\in B_L^c$, it follows then that
\begin{equation}
\begin{aligned}
 \int   (\rho_1 \wedge \rho_2)\lambda(\dx)   &\geq  \lambda_2(\mathbf{X}) - \frac{L-1}{L} \lambda_2( B_L) -
   \lambda_2( B_L^c)\geq   \lambda_2(\mathbf{X}) - \frac{L-1}{L} \lambda_2(\mathbf{X}) -
    \frac{K}{ L^p}\\&= \frac1L \lambda_2(\mathbf{X})  -\frac{K}{ L^p}\geq \Big(1-\frac1p\Big)\Big(\frac{\lambda_1(\mathbf{X})^p}{pK}\Big)^{\frac1{p-1}}>0,
  \end{aligned}
  \label{eq:almostThere}
\end{equation}
where the penultimate inequality is obtained by optimizing in $L$.
Recall now that we can always write $\mu_i(\dx)= \tilde{\rho}_i(x)
\tilde{\lambda}(\dx)$ for some measure $\tilde{\lambda}$ and functions
$\tilde{\rho}_i$ and moreover that
\begin{align*}
    \|\mu_1
  - \mu_2\|_{TV}   =  1- \int   (\tilde{\rho}_1 \wedge \tilde{\rho}_2)\tilde{\lambda}(\dx).  
\end{align*}
Since $\mu_i \ge \lambda_i$, it follows that $\|\mu_1 
  - \mu_2\|_{TV}  \leq 1 - \int   ({\rho}_1 \wedge
  {\rho}_2)(x)\lambda(\dx)$, which completes the proof in light of the bound in \eqref{eq:almostThere}.
\end{proof}

We now give the proof of Proposition~\ref{thm:Girsanov} using Lemma~\ref{prop:TV-two-measures}.
\begin{proof}[Proof of Proposition~\ref{thm:Girsanov}]
  We begin by defining, for measurable $A \subseteq \R^d$, 
  \begin{align*}
    \lambda_1(A) = \E \one_{W \in \mathcal{A}} \one_A(X_T(x,W)) \quad\text{and}\quad \lambda_2(A) = \E \one_{W \in \mathcal{A}} \one_A(X_T(\widetilde{x},W)), 
  \end{align*}
  and observe that $\delta_{x}P_T \geq \lambda_1$ and $\delta_{\widetilde{x}}P_T \geq \lambda_2$. Next, define $ \mathcal{M}_t = \exp ( Z_t -\tfrac12\langle Z
  \rangle_t)$, where $Z_t =  \int_0^t u_s(W) \dee W_s$ and $\langle Z
  \rangle_t$ is its quadratic variation. 
By {N}ovikov's theorem (see the proof by Krylov in \cite{Krylov}), we
know that $\mathcal{M}_t$ is a martingale with $\E \mathcal{M}_t=1$
for all $t \in[0,T]$. Since the quadratic variation $\langle Z
  \rangle_T =\int_0^T |u(s,W)|^2\dee s$ is less than $C$ by assumption, a straightforward
  application of It\^o's formula with standard a localization by
  stopping time argument  shows that $\E\mathcal{M}_T^2 \leq e^{2C}$. Recalling that
  $\P(\dee W)$ denotes the original Wiener measure on
  $C([0,T];\R^m)$, we define a  new measure on $C([0,T];\R^m)$ by
  $\mathbf{Q}(\dee W)=\mathcal{M}_T(W)\P(\dee W)$. Since we
  know that $\mathcal{M}_t$ is a martingale, Girsanov's Theorem \cite{Yor}
  implies that $\widetilde{W}_t= W_t + \int_0^t u(s,W)\ds$ is a standard Wiener process under the
  measure $\mathbf{Q}$. This means that 
  $\widetilde{X}_t(\widetilde{x},W) $ under the measure $\mathbf{Q}$ has the same distribution as $X_t(\widetilde{x},W)$
  under the measure $\P$. Hence, for any measurable $A \subseteq \R^d$,
  \begin{align*}
   \lambda_2(A)&= \E \one_{W \in \mathcal{A}} \one_A(X_T(\widetilde{x},W)) 
    = \E\one_{W \in \mathcal{A}}
                 \one_A(\widetilde{X}_T(\widetilde{x},W))
                 \mathcal{M}_t\\&=\E \one_{W \in \mathcal{A}}
    \one_A(X_T(x,W)) \mathcal{M}_t = \int_A  M(y) \lambda_1(\dy)
  \end{align*}
  where  $M(y)=\frac{d\lambda_2}{d\lambda_1}(y) =\E [
  \mathcal{M}_t | X_T(x,W)=y]$. This shows that
  $\lambda_1$ is equivalent to
  $\lambda_2$. Additionally, we have that
  \begin{align*}
    \int M(y)^2 \lambda_2(\dy) \le \E \mathcal{M}_T^2 \leq e^{2 C}
  \end{align*}
  We then apply Lemma~\ref{prop:TV-two-measures} to $\lambda_1$
  and $\lambda_2$ with $p=2$ to obtain the result.
\end{proof}

\subsection{Proof of Proposition~\ref{prop:T2control}}

With Propositions~\ref{thm:marginal measures} and~\ref{thm:Girsanov} in hand, the proof of Proposition~\ref{prop:T2control} is straightforward.

\begin{proof}[Proof of Proposition~\ref{prop:T2control}]
By Proposition~\ref{thm:Girsanov}, the control assumption \eqref{eq:controlprop} implies that there exists $\epsilon > 0$ such that for any $x,\tilde{x}, y \in \T$ we have 
\begin{equation} \label{eq:Pmarginal}
    \|\delta_{(x,y)}\Pt_{T_y}  - \delta_{(\tilde{x},y)}\Pt_{T_y}\|_{TV}\le \|\delta_{(x,y)}\Qt_{T_y}  - \delta_{(\tilde{x},y)}\Qt_{T_y}\|_{TV} \le 1-\epsilon, 
\end{equation}
where the first inequality is true because mapping two measures on $\R^2$ back to $\T^2$ can only reduce their total variation distance. The bound \eqref{eq:localdecayprop} then follows immediately by Proposition~\ref{thm:marginal measures}. As for \eqref{eq:globaldecayprop}, first observe that \eqref{eq:Pmarginal} and \eqref{eq:streamlineheuristic} imply that for $T = \sup_{y \in K} T_y$ we have
\begin{equation} \label{eq:globaldecayprop0}
    \|\Pt_T f\|_{L^\infty}\le (1-\epsilon)\|f\|_{L^\infty}.
\end{equation}
Iterating \eqref{eq:globaldecayprop0} and using that $\|\Pt_t f\|_{L^\infty} \le \|f\|_{L^\infty}$ for all $t \ge 0$ easily implies \eqref{eq:globaldecayprop}, which completes the proof.
\end{proof}

\begin{remark}
From the proof of Proposition~\ref{prop:T2control} given above, one can see that only the result of Section~\ref{sec:girsanov} is needed to obtain \eqref{eq:globaldecayprop} and hence prove Theorem~\ref{thm:T2general}. The result of Proposition~\ref{thm:marginal measures} on the convergence of general measures with the same $y$-marginal is fundamentally needed to obtain the global-in-time, streamline-by-streamline estimate of Theorem~\ref{thm:T2local}.
\end{remark}

In the case of radially symmetric shear flows on $\R^2$, the relevant SDE becomes 
\begin{equation} \label{eq:radialSDE}
\begin{cases}
    \dee \theta_t = -b(r_t)\dt + \frac{\sqrt{2\nu}}{r_t}\dee W_t, \\ 
    \dee r_t = \frac{\nu}{r_t}\dt + \sqrt{2\nu}\dee B_t.
\end{cases}
\end{equation}
As $r_t$ almost surely never hits zero for any $r_0 > 0$, \eqref{eq:radialSDE} can be view as an SDE on $\R^2$ and generates a well-defined flow map $\xi:\R_+ \times \R^2 \times \Omega \to \R^2$ that, as in \eqref{eq:periodizedtransition}, can be used to recover the Markov semigroup $\Pt_t$ for \eqref{eq:radialSDE} viewed as an equation for $(r_t,\theta_t) \in [0,\infty) \times [0,2\pi)$. Therefore, with a proof that is identical to the proof of Proposition~\ref{prop:T2control}, we obtain the following proposition, which will be used as the basis of our proof of Theorems~\ref{thm:RadialGen}.

\begin{proposition}\label{prop:R2control}
     Let $\xi\colon \R_+\times \R^2 \times \Omega \to \R^2$ and
     $\Pt_t$ be as defined above just after \eqref{eq:radialSDE} and
     for adapted controls $U_t,V_t:[0,\infty) \times \Omega \to \R$
     let $\tilde{\xi}\colon  \R_+\times \R^2 \times \Omega \to \R^2$ denote the flow map of the controlled SDE 
    \begin{equation} \label{eq:ControlledSDEradial}
    \begin{cases}
     \dee \tilde{\theta}_t = -b(\tilde{r}_t)\dt + \frac{\sqrt{2\nu}}{\tilde{r}_t}(\dee W_t + U_t \dt), \\ 
    \dee \tilde{r}_t = \frac{\nu}{r_t}\dt + \sqrt{2\nu}(\dee B_t + V_t \dt),
\end{cases}
    \end{equation}
    viewed as an equation on $\R^2$. Suppose that there are constants
    $c_0,C_0 > 0$ such that all $\theta,\tilde{\theta} \in [0,2\pi)$
    and $r \ge 0$ there exist $T_r < \infty$ and $U_t,V_t\colon [0,\infty) \times \Omega \to \R$ such that 
    \begin{equation} \label{eq:controlprop}
        \P \left(\int_0^\infty (|U_t|^2 + |V_t|^2) \dt \le C_0\right)=1 \quad \text{and} \quad \P(\xi(T_r,r,\theta,\omega) = \tilde{\xi}(T_r,r,\tilde{\theta},\omega)) \ge c_0.
    \end{equation}
    Then, there exists $C > 0$ such that for any $K \subseteq [0,\infty)$ and probability measures $\mu_1,\mu_2$ on $[0,\infty) \times [0,2\pi)$ that have the same $r$-marginal distributions and satisfy $\mu_i(K \times [0,2\pi)) = 1$, we have
\begin{equation} \label{eq:localdecayprop}
\|\mu_1\Pt_{T_*} - \mu_2 \Pt_{T_*}\|_{TV} \le (1-\epsilon)\|\mu_1 - \mu_2\|_{TV}, \quad \text{where} \quad T_* = \sup_{r \in K} T_r.
\end{equation}
 In particular, assuming that $T:= \sup_{r \in \T} T_r < \infty$ and setting $\lambda = 1/T$, there exist $c,C > 0$ such that for any $f_0 \in L^\infty([0,\infty)\times [0,2\pi))$ satisfying \eqref{eq:R2meanfree} the solution $\Pt_t f_0$ of \eqref{eq:ADEpolarshear} satisfies
\begin{equation}\label{eq:globaldecayprop}
        \|\Pt_t f\|_{L^\infty} \le C e^{-c\lambda t}\|f\|_{L^\infty}
    \end{equation}
    for all $t \ge 0$.
\end{proposition}

\begin{remark} \label{rem:bounded}
When \eqref{eq:ADEpolarshear} is posed on the unit disk, the associated SDE has a local time term and so the framework of Section~\ref{sec:girsanov} no longer immediately applies. However, the controls we construct will only cause trajectories to meet which never hit the boundary, and so it turns out that Proposition~\ref{thm:Girsanov} will still be sufficient. See Section~\ref{sec:radial2} for details.
\end{remark}

\section{Proof of Theorem~\ref{thm:T2general}} \label{sec:T2proof}

In this section we will prove Theorem~\ref{thm:T2general} by using its reduction to a quantitative control problem established in Section~\ref{sec:probability}. Throughout, $u: \T^2\to \R^2$ denotes a shear flow of the form \eqref{eq:T2shear} with $b\colon\T \to \R$ satisfying Assumption~\ref{Ass:T2general}, $\varphi:[0,h_0] \to [0,\infty)$ is the associated function defined in \eqref{eq:varphi}, and $T_\nu$ is as defined in the statement of Theorem~\ref{thm:T2general}. As in the statement of Assumption~\ref{Ass:T2general}, $\{\bar{y}_i\}_{i=1}^N$ and $\{y_i\}_{i=1}^M$ denote, respectively, the points where $b'$ is discontinuous or vanishes. Note that by virtue of $\lim_{\nu \to 0}\nu T_\nu = 0$ (already proven in Remark~\ref{rem:Tnu3}), we may always choose $\nu > 0$ small enough so that $\sqrt{\nu T_\nu} \ll 1$.

Before proceeding to the proof, we establish a bit of notation and record some basic properties of functions satisfying Assumption~\ref{Ass:T2general}. Let $\iota:\T \to \{1,-1\}$ be such that for any $y \in \T$ it holds that $h \mapsto b(y+\iota(y)h)$ is smooth on $(0,h_0)$, no critical points of $b$ lie strictly between $y$ and $y+\iota(y)h_0$, and 
$$ |b(y+\iota(y)h)-b(y)| \ge \varphi(h) \quad \forall h\in [0,h_0]. $$
We relabel the critical points $\{y_i\}_{i=1}^M$ of $b$ so that 
$$\{y_i\}_{i=1}^{m-1} = \{y \in \T: b^{(k)}(y) = 0 \quad \forall k\in \N\}$$ 
for some $1 \le m \le M+1$, with the understanding that $m = 1$ when the set is empty. For $0 < \delta \ll 1$ and $1 \le i \le m-1$, let $I_i(\delta) = (y_i - \delta, y_i + \delta)$. By Assumption~\ref{Ass:T2general}b, we may assume that $\delta \ll h_0$ is small enough so that for each $1\le i \le m-1$ both $h \mapsto |b'(y_i + h)|$ and $h \mapsto |b'(y_i-h)|$ are monotone increasing for $h \in [0,2\delta]$. For such a $\delta > 0$ fixed, we then define $t_\nu: \T \to (0,\infty)$ by 
\begin{equation} \label{eq:tnu}
    t_\nu(y)
    = 
    \begin{cases}
        \inf\{t \ge 0: t|b(y+\iota(y)\sqrt{\nu t})-b(y)| \ge 1\} & y \not \in \bigcup_{i=1}^{m-1} I_i(\delta) \\ 
        \inf\{t \ge 0: t|b(y_i)+\iota(y_i)\sqrt{\nu t})-b(y_i)| \ge 1\} & y \in I_i(\delta) \quad \text{for some} \quad 1\le i \le m-1. \\
    \end{cases}
\end{equation}
Note that $\sup_{y \in \T} t_\nu(y) \le T_\nu$ follows immediately from the definition. We also define 
\begin{equation} \label{eq:ellnu}
    \ell_\nu(y) = \sqrt{\nu t_\nu(y)},
\end{equation}
which satisfies $\sup_{y\in \T}\ell_\nu(y) \le \sqrt{\nu T_\nu} \ll 1$ for $\nu>0$ small and should be interpreted as the diffusive length scale associated with the timescale $t_\nu(y)$.  

The two technical lemmas below collect a few easy consequences of Assumption~\ref{Ass:T2general} that will be needed in the proof of Theorem~\ref{thm:T2general}. The proofs are tedious but completely elementary and will be deferred to Appendix~\ref{appendix}. We begin with a statement describing some properties of $b'$ and $t_\nu$ away from the set $\cup_{i=1}^{m-1}I_i(\delta)$. 

\begin{lemma}\label{lem:technical1}
 Let $\iota:\T \to \{1,-1\}$, $\{I_i(\delta)\}_{i=1}^{m-1}$, $t_\nu:\T \to (0,\infty)$, and $\ell_\nu(y) = \sqrt{\nu t_\nu(y)}$ be as defined above for a given function $b\colon\T \to \R$ satisfying Assumption~\ref{Ass:T2general}. Then, there exists $C \ge 1$ such that for all $\nu$ sufficiently small and $y \not \in \cup_{i=1}^{m-1}I_i(\delta)$ we have 
    \begin{equation} \label{eq:b'equivalence}
    \frac{1}{C} \le \inf_{h \in [1/2,4]} \frac{|b'(y+\iota(y)\ell_\nu(y)h)|}{|b'(y+\iota(y)\ell_\nu(y))|} \le \sup_{h\in [1/2,4]} \frac{|b'(y+\iota(y)\ell_\nu(y)h)|}{|b'(y+\iota(y)\ell_\nu(y))|}\le C \end{equation}
and 
\begin{equation}\label{eq:tnuboundregular}
    \frac{1}{C\ell_\nu(y)|b'(y+\iota(y)\ell_\nu(y))|} \le t_\nu(y) \le \frac{C}{\ell_\nu(y)|b'(y+\iota(y)\ell_\nu(y))|}.
\end{equation}
\end{lemma}
\noindent Thanks to Assumption~\ref{Ass:T2general}b, we also have a bound similar to \eqref{eq:tnuboundregular} that holds near the critical points where $b'$ vanishes to infinite order.

\begin{lemma}\label{lem:technical2}
Let $C \ge 1$ be as in Assumption~\ref{Ass:T2general}b. There exists $C_0 \ge 1$ such that for all $\nu$ sufficiently small and $y \in \cup_{i=1}^{m-1}I_i(\delta)$ there holds 
    \begin{equation} \label{eq:infiniteordertechnical}
       \frac{1}{\sqrt{\nu} |b'(y+\iota(y)C\ell_\nu(y))|} \le C_0t_\nu(y).
    \end{equation}
\end{lemma}

With the two lemmas above in hand we are now ready to prove Theorem~\ref{thm:T2general}. 

\begin{proof}[Proof of Theorem~\ref{thm:T2general}]
For $x_0,\tilde{x}_0, y_0 \in \T$ and a control $V_t$, let $(x_t,y_t)$ and $(\tilde{x}_t,\tilde{y}_t)$ denote the solutions of 
\begin{equation}\label{eq:sde_normal}
        \begin{cases}
            \dee x_t = -b(y_t)\,\dt + \sqrt{2\nu}\,\dee W_t, \\
            \dee y_t = \sqrt{2\nu}\,\dee B_t
        \end{cases} \quad \text{and} \qquad \begin{cases}
            \dee \tilde{x}_t = -b(\tilde{y}_t)\,\dt + \sqrt{2\nu}\,\dee W_t, \\
            \dee \tilde{y}_t = \sqrt{2\nu}\,\dee B_t + \sqrt{2\nu}V_t\,\dt, 
        \end{cases}
    \end{equation}
where $(x_t,y_t)|_{t=0} = (x_0,y_0)$ and $(\tilde{x}_t, \tilde{y}_t)|_{t=0} = (\tilde{x}_0, y_0)$. We view both of the SDEs above as equations on $\R^2$ with initial conditions in $[0,1)^2$. By Proposition~\ref{prop:T2control} and the fact that $\sup_{y \in \T} t_\nu(y) \le T_\nu$, it is sufficient to find constants $C,c > 0$ such that for any $x_0,\tilde{x}_0, y_0 \in \T$ and $\nu$ sufficiently small there exist $T \le Ct_\nu(y_0)$ and an adapted control $V_t$ satisfying
\begin{equation} \label{eq:generalthmgoal1}
\P\left(\int_0^T |V_t|^2 \dt \le C\right) = 1
\end{equation}
for which we have
\begin{equation} \label{eq:generalthmgoal2}
    \P\big(\,(x_{T}, y_{T}) = (\tilde{x}_{T}, \tilde{y}_{T})\,\big) \ge c.
\end{equation}
We will consider separately the cases $y_0 \not \in \cup_{i=1}^{m-1} I_i(\delta)$ and $y_0 \in \cup_{i=1}^{m-1} I_i(\delta)$.

\vspace{0.2cm}
\noindent \textbf{Case 1}: We suppose here that $y_0 \not \in \cup_{i=1}^{m-1} I_i(\delta)$. Without loss of generality, we assume $\iota(y_0) = 1$, $b'(y) \ge 0$ for all $y \in [y_0, y_0+h_0]$, and $x_0 < \tilde{x}_0$. For simplicity of notation, we write $\ell_\nu$ for $\ell_\nu(y_0)$ and $t_\nu$ for $t_\nu(y_0)$. Let
\begin{equation}\label{eq:stopping_time_initial_torus}
            \tau_0 = \inf\{t \ge 0 \colon y_t = y_0 + 2\ell_\nu\}
        \end{equation}
and
\begin{equation}\label{eq:stopping_time_final_torus}
            \overline{\tau} = \inf\{t \ge \tau_0 \colon y_t \not\in [y_0 + \ell_\nu, y_0 + 3\ell_\nu]\}.
        \end{equation}
The control will be chosen to be nonzero only for $\tau_0 \le t \le \bar{\tau}$ and the construction will proceed in two steps. First, we will choose $V_t$ so that $\tilde{y}_t - y_t$ grows enough that a sufficient differential is introduced in the velocities experienced by $x_t$ and $\tilde{x}_t$. Second, $V_t$ will be chosen to return $\tilde{y}_t - y_t$ to zero in such a way $x_t = \tilde{x}_t$ precisely when $y_t = \tilde{y}_t$. To ensure that \eqref{eq:generalthmgoal1} and \eqref{eq:generalthmgoal2} hold with $T \leqc t_\nu$, we will show that the length of each step is $\mathcal{O}(t_\nu)$ as $\nu \to 0$ and also that the cost of the associated control (i.e., its contribution to the left-hand side of \eqref{eq:generalthmgoal2}) is bounded uniformly in $\nu$. That $\bar{\tau} - \tau_0$ is large enough for our construction to take place with a probability that does not depend on $\nu$ will then follow from the scaling properties of Brownian motion.

\textbf{Step 1} (growing $\tilde{y}_t - y_t$): Let $\rho_t = \tilde{x}_t - x_t$ and $h_t = \tilde{y}_t - y_t$. For later use, we record the differentials
\begin{align}
    \dee h_t &= \sqrt{2\nu} V_t \dt, \label{eq:dh} \\ 
    \dee \rho_t & = -[b(y_t+h_t) - b(y_t)]\dt. \label{eq:drho}
\end{align}
Define the stopping time
$$ \tau_1 = \inf\{t \ge \tau_0: h_t = \ell_\nu \sqrt{\rho_t}\} \wedge \bar{\tau}. $$
For $\tau_0 \le t \le \tau_1$, we set $V_t = t_\nu^{-1/2}$. To estimate the length of the time interval $[\tau_0,\tau_1]$, first note that by \eqref{eq:dh} and the definition of $V_t$, we have 
\begin{equation}\label{eq:hstep1}
h_{t} = \frac{\sqrt{2\nu}}{\sqrt{t_\nu}}(t - \tau_0)
\end{equation}
for all for $\tau_0\le t \le \tau_1$. Now, by \eqref{eq:drho} and the fact that $b'(y) \ge 0$ for $y \in [y_0,y_0+h_0]$, $\rho_t$ is monotone decreasing for $\tau_0 \le t \le \tau_1$ with $0 < \rho_{\tau_0} \le 1$. Since also $h_t$ is monotone increasing for $\tau_0 \le t \le \tau_1$ with $h_{\tau_0} = 0$, it then follows by \eqref{eq:hstep1}, the definition of $\ell_\nu$, and the continuity of trajectories that we have
\begin{equation} \label{eq:taudiff1}
\tau_1 - \tau_0 \le t_\nu. 
\end{equation}
In particular, if $\bar{\tau} - \tau_0 > t_\nu$, then $\tau_1 \le \bar{\tau}$ and hence also $h_{\tau_1} = \ell_\nu \sqrt{\rho_{\tau_1}}$.

\textbf{Step 2} (decreasing $\tilde{y}_t - y_t$ while $\rho_t \to 0$): Define the stopping time 
$$ \tau_2 = \inf\{t \ge \tau_1:\rho_t = h_t = 0\} \wedge \bar{\tau}.$$
For $\tau_1 \leq t \leq \tau_2$, we set  
$$ V_t = -\frac{ \ell_\nu}{\sqrt{2\nu}} \left(\frac{b(y_t+h_t) - b(y_t)}{2 \sqrt{\rho_t}}\right).$$
Recalling \eqref{eq:drho}, this choice is made so that 
$$ \frac{dh_t}{dt} = \sqrt{2\nu}V_t = \ell_\nu \frac{d}{dt} \sqrt{\rho_t}, $$
which implies that $h_t = \ell_\nu\sqrt{\rho_t}$ for $\tau_1 \le t \le \tau_2$ whenever $\tau_1 \le \bar{\tau}$. Let us assume now in the computations for the remainder of this step that $\tau_1 \le \bar{\tau}$ and $\tau_1 \le t \le \tau_2$. By the definition of $\tau_0$ and $\bar{\tau}$, we have that $[y_t,y_t+h_t] \subseteq [y_0+\ell_\nu, y_0 + 4\ell_\nu]\subseteq [y_0,y_0+h_0]$, and hence by Lemma~\ref{lem:technical1} there exists $c_1 > 0$ such that 
$$ b(y_t+h_t) - b(y_t) \ge c_1 b'(y_0+\ell_\nu)h_t. $$
Thus, 
$$ \frac{d\rho_t}{dt} = -[b(y_t+h_t) - b(y_t)] \le -c_1 b'(y_0+\ell_\nu) \ell_\nu \sqrt{\rho_t}. $$
Solving this differential equality and then applying estimate \eqref{eq:tnuboundregular} of Lemma~\ref{lem:technical1} shows that 
\begin{equation}\label{eq:taudiff2}
    \tau_2 - \tau_1 \le \frac{\sqrt{2}}{c_1 \ell_\nu b'(y_0+\ell_\nu)} \le C_1 t_\nu
\end{equation}
for some $C_1 > 0$. The function $V_t$ is set to be zero for all $t \ge \tau_2$.

\textbf{Step 3} (cost of control and conclusion): We now estimate the integral in \eqref{eq:generalthmgoal2}. First note that by combining both inequalities in Lemma~\ref{lem:technical1}, it follows that for some $C_2 > 0$ we have
$$ |b(y_t+h_t) - b(y_t)| \leqc \sqrt{\rho_t}\ell_\nu(y)|b'(y_0+\ell_\nu)| \le \frac{C_2 \sqrt{\rho_t}}{t_\nu} \quad \forall \tau_1 \le t \le \tau_2.$$
Therefore, by \eqref{eq:taudiff1} and \eqref{eq:taudiff2} there holds
\begin{align*}
    \int_0^\infty |V_t|^2 \dt &\le \int_{\tau_0}^{\tau_1} \frac{1}{t_\nu} \dt + \frac{\ell_\nu^2}{8\nu}\int_{\tau_1}^{\tau_2} \Big|\frac{b(y_t+h_t) - b(y_t)}{\rho_t}\Big|^2 \dt \\ 
    & \le 1 + \frac{\ell_\nu^2}{8\nu} \frac{C_2^2}{t_\nu} = 1 + \frac{C_2^2}{8},
\end{align*}
which proves \eqref{eq:generalthmgoal2}. Regarding \eqref{eq:generalthmgoal1}, let $T = (2+C_1)t_\nu$ and 
$$\Omega_0 = \{\tau_0 \le t_\nu \text{ and } \tau_0 - \bar{\tau} \ge (1+C_1)t_\nu\}.$$
The construction of $V_t$ above implies that $(x_T(\omega),y_T(\omega)) = (\tilde{x}_T(\omega), \tilde{y}_T(\omega))$ for every $\omega \in \Omega_0$. The scaling properties of Brownian motion imply that $\P(\Omega_0)$ is bounded below uniformly in $\nu$, and so the proof under the assumption of Case 1 is now complete.

\vspace{0.2cm}
\noindent \textbf{Case 2}: We suppose now that $y_0 \in \cup_{i=1}^{m-1}I_i(\delta)$. The general scheme of the proof is the same as in Case 1, but we must proceed slightly differently in terms of how certain quantities are chosen to scale in $\nu$ because the estimates of Lemma~\ref{lem:technical1} are no longer true. As in Case 1, we may assume without loss of generality that $y_0 \in [y_i, y_i+\delta)$, $b'(y) \ge 0$ for all $y \in [y_0,y_0+h_0]$, and $x_0 < \tilde{x}_0$. Similar to before, we also define 
$$ \tau_0 = \inf\{t \ge 0: y_t = y_0 + 2C\ell_\nu\} \quad \text{and} \quad \bar{\tau} = \inf\{t \ge \tau_0: y_t \not \in [y_0 + C\ell_\nu, y_0 + 3C\ell_\nu]\}, $$
where $C \ge 1$ is as in Assumption~\ref{Ass:T2general}b.
For $\rho_t$ and $h_t$ as defined in Case 1, let 
$$ \tau_1 = \inf\{t\ge \tau_0: h_t = \sqrt{\nu}\sqrt{\rho_t}\} \wedge \bar{\tau}.$$
Then, we set $V_t = 1$ for $\tau_0 \le t \le \tau_1$. Similar to in the proof of Case 1, it is easy to check that 
\begin{equation}\label{eq:taudiff1case2}
    \tau_1 - \tau_0 \le 1.
\end{equation}
Next, let 
$$ \tau_2 = \inf\{t \ge \tau_1: \rho_t = h_t = 0\} \wedge \bar{\tau} $$
and for $\tau_1 \le t \le \tau_2$ set 
$$ V_t = -\left(\frac{b(y_t+h_t)-b(y_t)}{2\sqrt{2\rho_t}}\right).$$
Then, $h_t = \sqrt{\nu} \sqrt{\rho_t}$ for $\tau_1 \le t \le \tau_2$ whenever $\tau_1 \le \bar{\tau}$. Assuming that $\tau_1 \le \bar{\tau}$, for all $\tau_1 \le t \le \tau_2$ we have 
$$ \frac{d \rho_t}{dt} = -[b(y_t+h_t) - b(y_t)] \le -\sqrt{\nu} b'(y_0+C\ell_\nu) \sqrt{\rho_t}, $$
where in the inequality we used the definition of $\bar{\tau}$ and the fact that $y \mapsto b'(y)$ is monotone increasing for $y \in [y_0,y_0+h_0]$ by assumption. Solving the differential inequality and then applying Lemma~\ref{lem:technical2} shows that there exists $C_1 > 0$ such that  
\begin{equation}\label{eq:taudiff2case2}
    \tau_2 - \tau_1 \le C_1 t_\nu.
\end{equation}
As before we then set $V_t = 0$ for all $t \ge \tau_2$. By \eqref{eq:taudiff1case2} and \eqref{eq:taudiff2case2}, defining $I = [y_i-2\delta, y_i+2\delta]$ we have 
\begin{align*} \int_0^\infty |V_t|^2 \dt &\le \int_{\tau_0}^{\tau_1} \dt + \frac{1}{2}\int_{\tau_1}^{\tau_2} \frac{|b(y_t+h_t) - b(y_t)|^2}{\rho_t}\dt \\ 
& \le 1 + \frac{\nu}{2} \int_{\tau_1}^{\tau_2} \|b'\|^2_{L^\infty(I)}\dt \\ 
& \le 1+ \nu t_\nu C_1 \|b'\|^2_{L^\infty(I)}.
\end{align*}
As $\nu t_\nu$ remains bounded as $\nu \to 0$, we obtain \eqref{eq:generalthmgoal2}. The proof is then concluded as in Case 1 by setting $T = (2+C_1)t_\nu$.
\end{proof}

\section{Proofs for radial shears on $\R^2$ and the unit disk} \label{sec:Radialproofs}

In this section, we will prove Theorem~\ref{thm:RadialGen} by slightly modifying the control argument used in the previous section to prove Theorem~\ref{thm:T2general}. The main approach will be essentially unchanged, but a few new ingredients will be needed to deal with the fact that the SDE associated with \eqref{eq:ADEpolarshear} has multiplicative noise due to the variable coefficient Laplacian. We will begin in Section~\ref{sec:radial1} by proving Theorem~\ref{thm:RadialGen} in the $\R^2$ setting. Then, in Section~\ref{sec:radial2} we discuss the minor modifications necessary to treat shear flows on the unit disk. 

\subsection{Proof of Theorem~\ref{thm:RadialGen} on $\R^2$} \label{sec:radial1}

Throughout this section, $b\colon[0,\infty) \to \R$ is a function
satisfying Assumption~\ref{ass:R2_general}, $\{r_1,\ldots,r_M\}
\subseteq (0,\infty)$ denote its critical points, and $q \ge 0$ is
such that \eqref{eq:qdef} holds. Moreover, $t_\nu\colon  [0,\infty) \to
(0,\infty)$ is as defined in \eqref{eq:tnurintro} and $\iota\colon
[0,\infty) \to \{1,-1\}$ is such that, for some $0 < h_0 \ll 1$ that
depends only on $b$, no elements of $\{0,r_1,\ldots,r_M\}$ lie
strictly between $r$ and $r+\iota(r)h_0$ and $h \mapsto
|b(r+\iota(r)h) - b(r)|$ is monotone increasing for $h \in
[0,h_0]$. As on $\T^2$, $\iota(r)$ is not necessarily
unique when $r$ is far from $\{0,r_1,\ldots,r_M\}$ and in this case we may simply set $\iota(r)$ to be the choice that gives a smaller value in \eqref{eq:tnurintro}. Just like in Section~\ref{sec:T2proof}, we also define the length scale $\ell_\nu(r) = \sqrt{\nu t_\nu(r)}$ and observe that we may always assume that $\sup_{r \in [0,\infty)} \ell_\nu(r)$ is as small we wish due to \eqref{eq:tnurbound}.

We begin with a technical lemma that will play the role that Lemma~\ref{lem:technical1} played in the proof of Theorem~\ref{thm:T2general}. Just like Lemmas~\ref{lem:technical1} and~\ref{lem:technical2} the proof is completely elementary and will be deferred to Appendix~\ref{appendix}.

\begin{lemma} \label{lem:technicalradial}
    Let $b\colon [0,\infty) \to \R$ satisfy Assumption~\ref{ass:R2_general}. There exists $C \ge 1$ such that for all $\nu$ sufficiently small and $r \in [0,\infty)$ we have 
    \begin{equation} \label{eq:b'equivalencerad}
    \frac{1}{C} \le \inf_{h \in [1/2,6]} \frac{|b'(r+\iota(r)\ell_\nu(r)h)|}{|b'(r+\iota(r)\ell_\nu(r))|} \le \sup_{h\in [1/2,6]} \frac{|b'(r+\iota(r)\ell_\nu(r)h)|}{|b'(r+\iota(r)\ell_\nu(r))|}\le C \end{equation}
and
\begin{equation}\label{eq:tnuboundregularrad}
    \frac{1}{C\ell_\nu(r)|b'(r+\iota(r)\ell_\nu(r))|} \le t_\nu(r) \le \frac{C}{\ell_\nu(r)|b'(r+\iota(r)\ell_\nu(r))|},
\end{equation}
\end{lemma}

We are now ready to prove Theorem~\ref{thm:RadialGen} in the $\R^2$ case. The construction of the control follows an argument very similar to the proof of Theorem~\ref{thm:T2general}.

 \begin{proof}[Proof of Theorem~\ref{thm:RadialGen} when $D = \R^2$]

 For adapted controls $U_t,V_t\colon[0,\infty) \to \R$, let $(r_t,\theta_t)$ and $(\tilde{r}_t, \tilde{\theta}_t)$ denote the solutions of the SDEs
\begin{equation} \label{eq:radialSDEs}
\begin{cases}
    \dee \theta_t = -b(r_t)\dt + \frac{\sqrt{2\nu}}{r_t}\dee W_t, \\ 
    \dee r_t = \frac{\nu}{r_t}\dt + \sqrt{2\nu}\dee B_t
\end{cases}
\quad \text{and} \quad 
\begin{cases}
     \dee \tilde{\theta}_t = -b(\tilde{r}_t)\dt + \frac{\sqrt{2\nu}}{\tilde{r}_t}(\dee W_t + U_t \dt), \\ 
    \dee \tilde{r}_t = \frac{\nu}{\tilde{r}_t}\dt + \sqrt{2\nu}(\dee B_t + V_t \dt)
\end{cases}
\end{equation}
viewed as equations on $\R^2$ (recall the discussion preceding Proposition~\ref{prop:R2control}) and with initial conditions $(r_0,\theta_0), (\tilde{r}_0, \tilde{\theta}_0) \in (0,\infty) \times [0,2\pi)$. By Proposition~\ref{prop:R2control}, it is sufficient to show that there exist constants $c, C > 0$ such that for any pair of initial conditions $(r_0,\theta_0), (\tilde{r}_0, \tilde{\theta}_0) \in (0,\infty) \times [0,2\pi)$ there exist $T \leq Ct_\nu(r_0)$ and a choice of controls $U_t,V_t$ such that 
\begin{equation}\label{eq:radial_cost}
   \P\Big(\int_0^T (|(U_t|^2 +|V_t|^2)\dt \le C\Big) = 1
        \end{equation}
and 
\begin{equation} \label{eq:radialhit}
\P((r_T, \theta_T) = (\tilde{r}_T, \tilde{\theta}_T))\ge c.
\end{equation}

Before proceeding to the proof, we record an inequality that will be needed. Specifically, there is $c > 0$ such that for all $r \in [0,\infty)$ we have 
\begin{equation} \label{eq:ellnuass}
   \frac{c}{r} \le \frac{1}{r+\iota(r)\ell_\nu(r)h} \le \frac{2}{r+\ell_\nu(r)} \quad \forall h\in [1,6]
\end{equation}
provided that $\nu$ is sufficiently small. Regarding the upper bound, for fixed $0 < \delta \ll h_0$ and all $r \in [0,\delta]$ it is trivial because in this case $\iota(r) = 1$. On the other hand, when $r > \delta$, the upper bound clearly holds for all $\nu$ sufficiently small from the fact that $\lim_{\nu \to 0} \sup_{r \ge 0}\ell_\nu(r) = 0$. The lower bound is immediate for similar reasons when $r > \delta$, and for $r \le \delta$ it holds from the easily verified fact that $\ell_\nu(r) \le C r$ for all $r \in [0,\delta]$.

With \eqref{eq:ellnuass} in hand, we now turn to the control argument. We may assume without loss of generality that $\theta_0 < \tilde{\theta}_0$, $\tilde{\theta}_0 - \theta_0 \le \pi$, and $b'(r) \ge 0$ for all $r$ between $r_0$ and $r_0+ \iota(r_0)h_0$. For simplicity of notation, we will write just $\ell_\nu$ and $t_\nu$ for the functions evaluated at $r = r_0$. Let $\rho_t = \tilde{\theta}_t - \theta_t$ and $h_t = \iota(r_0)(\tilde{r}_t - r_t)$. For future use, we record the differentials  
\begin{equation}\label{eq:rho_radial_differential}
            \dee \rho_t = -[b(\tilde{r}_t) - b(r_t)]\dt - \sqrt{2\nu} \frac{h_t}{r_t \tilde{r}_t} \dee W_t + \sqrt{2\nu}\frac{U_t}{ \tilde{r}_t} \dt
        \end{equation}
and 
\begin{equation}\label{eq:h_radial_differential}
            \dee h_t = -\nu\frac{\iota(r_0) h_t}{r_t\tilde{r}_t} \dt + \iota(r_0)\sqrt{2\nu}V_t\dt.
        \end{equation}
 We also define the stopping times 
\begin{equation}\label{eq:on_switch_radial}
\tau_0 = \inf\{t > 0 \colon r_t  = r_0 + 2\iota(r_0)\ell_\nu\}
\end{equation}
and
\begin{equation}\label{eq:off_switch_radial}
            \overline{\tau} = \inf\{t \ge \tau_0 \colon |r_t - (r_0 + 2\iota(r_0) \ell_\nu)| \ge \ell_\nu\},
        \end{equation}
and we set $U_t = V_t = 0$ for all $t \le \tau_0$ and $t \ge \bar{\tau}$.

\textbf{Step 1} (increasing $h_t$): Let
\begin{equation}\label{eq:tau_1_radial}
\tau_1 = \inf\{t > \tau_0 \colon h_t = \ell_{\nu}\sqrt{\rho_t}\} \wedge (\tau_0 + \sqrt{\pi}\delta t_\nu) \wedge \bar{\tau},
\end{equation}
where $\delta > 0$ is a small parameter to be chosen at the end of the proof. For $\tau_0 \le t \le \tau_1$, we set $U_t = 0$ and 
\begin{equation}\label{eq:r_control_radial_step_1}
V_t = \sqrt{\frac{\nu}{2}}\frac{h_t}{r_t\tilde{r}_t} + \frac{\iota(r_0)}{\delta\sqrt{t_\nu}}.
\end{equation}
Then, by \eqref{eq:h_radial_differential} and \eqref{eq:r_control_radial_step_1}, for $\tau_0 \le t \le \tau_1$ we have
\begin{equation}\label{eq:h_radial_step_1}
h_t = \frac{\sqrt{2\nu}}{\delta\sqrt{t_\nu}}(t - \tau_0).
\end{equation}
In particular, $h_t$ increases in a continuous manner from zero to $\sqrt{2\pi} \ell_\nu$ over time window $\tau_0 \le t \le \tau_1$, unless the equality $h_t = \ell_\nu \sqrt{\rho_t}$ occurs sooner than $t = \tau_0 + \sqrt{\pi}\delta t_\nu$.
Let 
\begin{equation} \label{eq:Omega1def}
\Omega_1 = \Big\{\omega \in \Omega: \sup_{\tau_0 \le t \le 
\tau_1} \rho_t \le 2 \pi\Big\}. 
\end{equation}
Since $\rho_{\tau_0} > 0$, it is easy to see from \eqref{eq:h_radial_step_1} that \begin{equation} \label{eq:Omega1purpose}
h_{\tau_1} = \ell_\nu \sqrt{\rho_{\tau_1}} \quad \forall \omega \in \Omega_1 \cap \{\bar{\tau} - \tau_0 \ge \sqrt{\pi}\delta t_\nu\}.
\end{equation} 
For later use in deducing \eqref{eq:radialhit} we now estimate $\P(\Omega_1)$ from below. First observe that by \eqref{eq:rho_radial_differential} and $\tilde{r}_t = r_t + \iota(r_0) h_t$, for $\tau_0 \leq t \leq \tau_1$ we have  
\begin{equation}\label{eq:step_one_rho_formula}
\rho_t = \rho_{\tau_0} - \int_{\tau_0}^t[b(r_s + \iota(r_0)h_s) - b(r_s)]\dee s - M_t, \quad \text{where} \quad M_t = \sqrt{2\nu} \int_{\tau_0}^t \frac{h_s}{r_s (r_s + \iota(r_0)h_s)} \dee W_s.
\end{equation}
Since $0<\rho_{\tau_0} \le \pi$ and $b'(r) \ge 0$ for $r \in [r_0 + \iota(r_0)h_0]$, we see then that $\rho_t \le \pi + |M_t|$
for $\tau_0 \le t \le \tau_1$ and to have $\omega \in \Omega_1$ it suffices for $\sup_{\tau_0 \le t \le \tau_1} |M_t| \le \pi$ to hold. Notice now that since $0 \le h_s \le \sqrt{2\pi}\ell_\nu$ and $|r_t - (r_0 + 2\iota(r_0) \ell_\nu)| \le \ell_\nu$ for $\tau_0 \le t \le \tau_1$, we see that $r_t$ and $r_t + \iota(r_0) h_t$ lie between $r_0$ and $r_0 + 6\iota(r_0) \ell_\nu$ for all $\tau_0 \le t \le \tau_1$. Thus, applying \eqref{eq:ellnuass}, we have  
\begin{equation} \label{eq:rdenominatorclarify}
    \left|\frac{1}{r_t (r_t+ \iota(r_0) h_t)}\right|\le \frac{4}{\ell_\nu^2} \quad \forall t\in[\tau_0,\tau_1].
\end{equation}
By Doob's martingale inequality, the It\^{o} isometry, and \eqref{eq:rdenominatorclarify} we then have 
\begin{align*}
   \P\Big(\sup_{\tau_0 \le t \le \tau_1^*} |M_t| > \pi\Big)  \le \frac{2 \nu}{\pi^2} \E \int_{\tau_0}^{\tau_1} \Big|\frac{h_t}{r_t (r_t + \iota(r_0)h_t)}\Big|^2 \dee t  
 \le \frac{64 \nu}{\pi} \E \int_{\tau_0}^{\tau_1} \frac{1}{\ell_\nu^2} \dee t \le 128\delta,
\end{align*}
where in the last inequality we used that by definition we have $\nu t_\nu/ \ell_\nu^2 = 1$. Thus,
\begin{equation} \label{eq:Omega1prob}
    \P(\Omega_1) \ge 1 - 128 \delta.
\end{equation}

\textbf{Step 2} (decreasing $h_t$ while $\rho_t$ gets close to zero): This step proceeds quite similarly to Step 2 in the proof of Theorem~\ref{thm:T2general}, but due to the multiplicative noise we will only bring $\rho_t$ to zero up to a small error term that will be dealt with below in Step 3. For $t \geq \tau_1$, we define the process  
\begin{equation}\label{eq:eta}
            \begin{cases}
                \dee \eta_t = -\frac{1}{\delta^2}[b(\tilde{r}_t)- b(r_t)]\dt \\
                \eta_{\tau_1} = \frac{1}{\delta^2}\rho_{\tau_1},
            \end{cases}
        \end{equation}
where $\delta > 0$ is as in Step 1. We introduce the stopping time 
\begin{equation}\label{eq:stopping_radial_step_2}
 \tau_2 = \inf\{t \ge \tau_1 \colon \eta_t = h_t = 0 \} \wedge \bar{\tau}.
\end{equation}
Let $\tilde{\Omega}_1 = \{h_{\tau_1} = \ell_\nu \sqrt{\rho_{\tau_1}}\} \subseteq \Omega$. For $\tau_1\leq t\leq \tau_2$, we set $U_t = 0$ and 
\begin{equation}\label{eq:V_control_radial_step_2}
 V_t(\omega) = \mathbf{1}_{\tilde{\Omega}_1}\times \left(\sqrt{\frac{\nu}{2}}\frac{h_t}{r_t\tilde{r}_t} - \frac{1}{\delta} \frac{\iota(r_0) \ell_\nu}{2 \sqrt{2\nu \eta_t}}[b(\tilde{r}_t) - b(r_t)]\right).
\end{equation}
This choice implies that $h_t = \delta \ell_\nu \sqrt{\eta_t}$ for all $\tau_1 \le t \le \tau_2$ whenever $\omega \in \tilde{\Omega}_1$. It then follows by \eqref{eq:eta} that for $\omega \in \tilde{\Omega}_1$ and all $\tau_1 \le t \le \tau_2$ we have 
\begin{equation} \label{eq:eta2}
\dee \eta_t = -\frac{1}{\delta^2}[b(r_t + \iota(r_0)h_t) - b(r_t)]\dt. 
\end{equation}
Now, since $h_{\tau_1} \le \sqrt{2\pi} \ell_\nu$ from the construction in Step 1 and the choice of $V_t$ in this step ensures $h_t$ is monotone decreasing for $\tau_1 \le t \le \tau_2$, we know that $h_t$ in \eqref{eq:eta2} satisfies $0 \le h_t \le \sqrt{2\pi} \ell_\nu$. Thus, we may apply both inequalities of Lemma~\ref{lem:technicalradial} in \eqref{eq:eta2} to conclude that for some $c_1 > 0$ there holds 
\begin{equation}
    \frac{\dee \eta_t}{\dee t} \leqc - \frac{\ell_\nu |b'(r_0 + \iota(r_0) \ell_\nu)|}{\delta} \sqrt{\eta_t} \le - \frac{c_1}{ \delta t_\nu} \sqrt{\eta_t}. 
\end{equation}
Solving this differential inequality and recalling \eqref{eq:Omega1purpose}, we have shown that there exists $C_1 > 0$ such that 
\begin{equation}\label{eq:t2-t1}
\tau_2 - \tau_1 \le C_1 \delta t_\nu \quad \text{and} \quad h_{\tau_2} = \eta_{\tau_2} = 0 \quad \forall \omega \in \Omega_1 \cap \{\bar{\tau} - \tau_0 \ge \delta (\sqrt{\pi}+C_1) t_\nu\}.
\end{equation}

\textbf{Step 3} (removing error between $\eta_{\tau_2}$ and $\rho_{\tau_2}$): Let
$$ \Omega_2 = \{ |\rho_{\tau_2}| \le \ell_\nu/r_0\} \cap \{h_{\tau_2} = 0\} \subseteq \Omega$$ and define the stopping time
$$ \tau_3 = \inf\{t \ge \tau_2: \rho_t = h_t = 0\} \wedge \bar{\tau}.$$
For $\tau_2 \le t \le \tau_3$ we set $V_t = 0$ and 
\begin{equation}
    U_t = \mathbf{1}_{\Omega_2} \frac{\mathrm{sgn}(\rho_{\tau_2})}{\sqrt{t_\nu}}.
\end{equation}
Then, since $h_t = 0$ for $t \ge \tau_2$ on the set $\Omega_2$, it follows from \eqref{eq:rho_radial_differential} that for $\omega \in \Omega_2$ and $\tau_2 \le t \le \tau_3$ we have
$$ \rho_t = \rho_{\tau_2} + \sqrt{\frac{2\nu}{t_\nu}} \int_{\tau_2}^t \frac{\mathrm{sgn}(\rho_{\tau_2})}{r_s} \dee s.  $$
From the lower bound in \eqref{eq:ellnuass} and the fact that $|\rho_{\tau_2}| \le \ell_\nu/r_0$ for $\omega \in \Omega_2$, it follows easily that for some $C_2 > 0$ we have
\begin{equation} \label{eq:t3-t2}
    \tau_3 - \tau_2 \le C_2 t_\nu \quad \text{whenever} \quad \omega \in \Omega_2.
\end{equation}
We set both controls to be zero for $t \ge \bar{\tau} \wedge \tau_3$.

\textbf{Step 4} (conclusion and cost of control): Let $T = (1+\sqrt{\pi} \delta + C_1 \delta + C_2)t_\nu$ and define 
\begin{equation} \label{eq:Omega0}
  \Omega_0 = \{\tau_0 \le t_\nu\} \cap \{\bar{\tau} - \tau_0 \ge (T-1)t_\nu\}.
\end{equation}
In view of \eqref{eq:Omega1purpose}, \eqref{eq:t2-t1}, and \eqref{eq:t3-t2}, the construction carried out in Steps 1-3 is such that
\begin{equation} 
(r_T, \theta_T) = (\tilde{r}_T, \tilde{\theta}_T) \quad \forall \omega \in \Omega_0 \cap \Omega_1 \cap \{|\rho_{\tau_2}| \le \ell_\nu/r_0\}.
\end{equation}
To estimate this probability, first recall that the procedure in Step 2 was such that for $\omega \in \Omega_0 \cap \Omega_1$ we have $\eta_{\tau_2} = 0$, $\tau_2 - \tau_1 \le C_1 \delta t_\nu$ and $0 \le h_t \le \sqrt{2\pi} \ell_\nu$ for all $\tau_1 \le t \le \tau_2$. Therefore, since $$\rho_{t} = \eta_{t} - \sqrt{2\nu}\int_{\tau_1}^t \frac{h_t}{r_t \tilde{r}_t}\dee W_t $$
for $\tau_1 \le t \le \tau_2$ and \eqref{eq:rdenominatorclarify} still holds for $\tau_1 \le t \le \tau_2$, by carrying out a computation similar to the estimate of $M_t$ from Step 1, we have 
\begin{align*}
    \E \mathbf{1}_{\Omega_0 \cap \Omega_1}|\rho_{\tau_2}|^2 &= 2 \nu \E \mathbf{1}_{\Omega_0 \cap \Omega_1} \left|\int_{\tau_1}^{\tau_2 \wedge (\tau_1 + C_1 \delta t_\nu)} \frac{\mathbf{1}_{|h_t| \le \sqrt{2\pi} \ell_\nu}h_t}{r_t \tilde{r}_t}\dee W_t \right|^2 \\ 
    & \le 2\nu \E \int_{\tau_1}^{\tau_2 \wedge (\tau_1 + C_1 \delta t_\nu)}\mathbf{1}_{|h_t| \le 4\pi \ell_\nu} \left|\frac{h_t}{r_t \tilde{r}_t}\right|^2 \dee t \\
    & \le 128 \pi \frac{\nu \ell_\nu^2}{(\ell_\nu + r_0)^4} \E\int_{\tau_1}^{\tau_2 \wedge (\tau_1 + C_1 \delta t_\nu)} \dt    \\
    & \le 128 \delta \pi C_1 \left(\frac{\ell_\nu}{\ell_\nu + r_0}\right)^4.
\end{align*}
It follows by Chebyshev's inequality that 
\begin{equation} \label{eq:Omega12}
    \P\left(\Omega_0 \cap \Omega_1 \cap \{|\rho_{\tau_2}| \le \ell_\nu/r_0\}\right) \ge \P(\Omega_0 \cap \Omega_1) - 128 \delta \pi C_1.
\end{equation}
Since $\P(\Omega_0)$ is bounded below uniformly in $\nu,\delta \in (0,1]$, we see that \eqref{eq:radialhit} holds when $\delta > 0$ is sufficiently small by \eqref{eq:Omega1prob} and \eqref{eq:Omega12}.

With $\delta > 0$ fixed so that \eqref{eq:radialhit} holds, it remains to check \eqref{eq:radial_cost} with $T$ as defined at the beginning of the previous paragraph. From the choice of controls in Steps 1-3, we have
        $$ \int_{0}^{T} (|U_t|^2 + |V_t|^2)\dt \le I_1 + I_2 + I_3, $$
where
$$ I_1 = \int_{\tau_0}^{\tau_1}\left( \sqrt{\frac{\nu}{2}} \frac{h_t}{r_t \tilde{r}_t} + \frac{1}{\delta \sqrt{t_\nu}}\right)^2\dt, $$
$$ I_2 = \mathbf{1}_{\{h_{\tau_1} = \ell_\nu \sqrt{\rho_{\tau_1}}\}}\times\int_{\tau_1}^{\tau_2} \left( \sqrt{\frac{\nu}{2}} \frac{h_t}{r_t \tilde{r}_t} + \frac{\sqrt{t_\nu}}{\delta \sqrt{\eta_t}} |b(\tilde{r}_t) - b(r_t)|\right)^2\dt, $$
and 
$$ I_3 = \mathbf{1}_{\Omega_2}\times\int_{\tau_2}^{\tau_3} \frac{1}{t_\nu}\dt. $$
From \eqref{eq:t3-t2}, it is clear that $I_3 \le C_2$. Next, the bound $I_1 \le C$ for some $C > 0$ follows easily from $\tau_1 \le \tau_0 + \sqrt{\pi} \delta t_\nu$, the upper bound in \eqref{eq:ellnuass}, and the fact that $|h_t| \le \sqrt{2\pi} \ell_\nu$ for $\tau_0 \le t \le \tau_1$. We lastly turn to $I_2$. By the construction in Step 2, on the set $\{h_{\tau_1} = \ell_\nu \sqrt{\rho_{\tau_1}}\} \subseteq \Omega$ we have $\tau_2 \le \tau_1 + C_1 \delta t_\nu$ and $0 \le h_t = \delta \ell_\nu \sqrt{\eta_t} \le \sqrt{2 \pi} \ell_\nu$ for all $\tau_1 \le t \le \tau_2$. The first term in $I_2$ is thus bounded in the same way as the corresponding term in $I_1$. For the second piece of $I_2$, we appeal to Lemma~\ref{lem:technicalradial} to obtain  
\begin{align*}
    \mathbf{1}_{\{h_{\tau_1} = \ell_\nu \sqrt{\rho_{\tau_1}}\}} \int_{\tau_1}^{\tau_2} \frac{t_\nu}{\delta^2 \eta_t} |b(r_t + \iota(r_0) h_t) - b(r_t)|^2 \dt & \le C t_\nu \int_{\tau_1}^{\tau_2 + C_1 \delta t_\nu} \ell_\nu^2 |b'(r_0 + \iota(r_0)\ell_\nu)|^2 \dt  \\ 
    & \le \frac{C}{t_\nu} \int_{\tau_1}^{\tau_2 + C_1 \delta t_\nu} \dt \le C.
\end{align*}
This proof is now complete. 

\end{proof}

\subsection{Modifications on the unit disk} \label{sec:radial2}

We now discuss the modifications of the proof given in Section~\ref{sec:radial1} needed to prove Theorem~\ref{thm:RadialGen} on the unit disk. The Neumann boundary condition in \eqref{eq:ADEpolarshear} corresponds to the condition that the diffusion $r_t$ in \eqref{eq:radialSDE} is reflected at the boundary $r = 1$ and is enforced by including a local time term in the SDE. In particular, the relevant SDE becomes 
\begin{equation} \label{eq:SDEdisk}
    \begin{cases}
        \dee \theta_t = -b(r_t)\dt + \frac{\sqrt{2\nu}}{r_t}\dee W_t, \\ 
    \dee r_t = \frac{\nu}{r_t}\dt + \sqrt{2\nu}\dee B_t  - \frac{1}{2}\dee L_t,
    \end{cases}
\end{equation}
where $L_t$, the local time of diffusion $r_t$ at $r = 1$, is a continuous, non-decreasing process satisfying $L_0 = 0$ and 
\begin{align*}
  L_t = \int_0^t \mathbf{1}_{\{r_s = 1\}} \dee L_s.
\end{align*}
See \cite{Stroock_Varadhan_1971,AndersonOrey,Tanaka_1979,Potsdam} for more background on SDEs with reflections realized
through a Skorokhod Problem using local times.
While we did not include SDEs with a local time term in Section~\ref{sec:probability}, the modifications from the $\R^2$ case needed to treat SDEs on the disk are actually quite minimal. The key point here is that the use of $\bar{\tau}$ in our proof on $\R^2$ implies that if the two initial conditions in \eqref{eq:SDEdisk} lie on a streamline away from a suitable diffusive layer near $r = 1$, then the control argument from Section~\ref{sec:radial1} immediately generalizes to couple trajectories that never reach the boundary with positive probability. This observation will be enough for the results of Section~\ref{sec:girsanov} to be sufficient. We will just need a statement that says solutions of \eqref{eq:SDEdisk} that start close to the boundary escape far enough into the bulk on an appropriate timescale, which is provided by the following lemma. 

\begin{lemma} \label{lem:escape} 
    Fix $0<\ell < 1/4$ and define $T$ through the relation $\ell = \sqrt{\nu T}$. There exist constants $C,c > 0$ that do not depend on $\nu$ or $\ell$ such that for every $r_0 \in [1-\ell,1]$ the solution $r_t$ of \eqref{eq:SDEdisk} with initial condition $r_0$ satisfies 
    $$ \P(1-3\ell \le r_{CT} \le 1-\ell) > c. $$
\end{lemma}

\begin{proof}
For a general initial condition $r_0 \in (0,1]$ and $r_t$ the associated solution of the second equation in \eqref{eq:SDEdisk}, define the two stopping times 
\begin{equation} \label{eq:tau1hit}
    \tau_1(r_0) = \inf\{t \ge 0: r_t = 1-2\ell\} 
\end{equation}
and
\begin{equation} \label{eq:tau2hit}
    \tau_2(r_0) = \inf\{t \ge 0: r_t \not \in (1-3\ell, 1-\ell)\}. 
\end{equation}
We will show below that there exist $C_1,c_1 > 0$ that do not depend on $\nu$ or $\ell$ such that
\begin{equation} \label{eq:hitbound}
     \P\big(\tau_1(r_0) \le C_1T\big) > c_1 \quad \forall r_0 \in [1-\ell, 1].
\end{equation}
For now, we complete the proof assuming \eqref{eq:hitbound}. By the scaling properties of Brownian motion, it is easy to see that for some $c_2 > 0$ we have 
$$ \P(\tau_2(1-2\ell) \ge 4C_1 T) \ge c_2$$
Thus, for any $r_0 \in [1-\ell,1]$ we have 
$$ \P(1-3\ell \le r_{2C_1T} \le 1-\ell) \ge \P(\tau_1(r_0) \le C_1 T)\P(\tau_2(1-2\ell) \ge 4C_1T) \ge c_1 c_2 > 0.  $$

It remains to prove \eqref{eq:hitbound}. Consider the problem to find $\psi:[1-2\ell,1] \to \R$ such that 
\begin{equation} \label{eq:hittingPDE}
\begin{cases}
\nu \left(\frac{1}{r}\partial_r \psi + \partial_{rr} \psi\right) = -1 & \text{in }(1-2\ell,1), \\
\psi(1-2\ell) = 0, \\
\partial_r \psi(1) = 0.
\end{cases}
\end{equation}
Using the Lax-Milgram theorem, one can construct a unique $\psi \in H^1([1-2\ell,1];\R)$ that solves the variational formulation of \eqref{eq:hittingPDE}. By elliptic regularity, this unique solution is actually in $H^2$, and then from Sobolev embedding and the PDE itself one sees that $\psi \in C^2([1-2\ell,1];\R)$ is actually a classical solution of \eqref{eq:hittingPDE}. It is well known (see e.g. \cite[Theorem 3.1.3]{Potsdam}) that this classical solution is connected to \eqref{eq:SDEdisk} via the relationship 
\begin{equation} \label{eq:psihit}
\psi(r) = \E \tau_1(r).
\end{equation}
To estimate $\psi$, we define the new variable $$\tilde{r} = \frac{r-(1-2\ell)}{2\ell}$$ and let $\Psi(\tilde{r}) = \frac{\psi(r)}{4T}.$ 
Defining $c_\ell = \frac{1}{2\ell} - 1$, it is easy to check that $\Psi:[0,1] \to \R$ solves 
\begin{equation}\label{eq:rescaledhittingPDE}
\begin{cases}
  \frac{1}{\tilde{r} + c_\ell} \partial_{\tilde{r}}\Psi + \partial_{\tilde{r}\tilde{r}} \Psi = -1 & \text{in }(0,1), \\
  \Psi(0) = 0, \\
  \partial_{\tilde{r}}\Psi(1) = 0.
  \end{cases}
\end{equation}
The first equation above can be re-written as 
$$ \partial_{\tilde{r}}\big((\tilde{r} + c_\ell)\partial_{\tilde{r}}\Psi\big) = -(\tilde{r} + c_\ell) $$
from which a simple energy estimate gives 
$$ c_\ell \|\partial_{\tilde{r}}\Psi\|_{L^2([0,1])}^2 \le (1+c_\ell)\|\Psi\|_{L^2([0,1])}. $$
By the Poincar\'{e} inequality and the fact that $c_\ell \ge 1$, it follows that $\|\partial_{\tilde{r}} \Psi\|_{L^2([0,1])} \le C_2$ for a constant $C_2$ that is independent of $\nu$ and $\ell$. Thus, by the 1d Sobolev embedding, $\|\psi\|_{L^\infty([1-\ell,1])} \le C_3 T$ for some $C_3 \ge 1$ and the desired estimate then follows by \eqref{eq:psihit} and Chebyshev's inequality.  
\end{proof}

With Lemma~\ref{lem:escape} in hand, the Theorem~\ref{thm:RadialGen} on the disk is an easy corollary of the proof given in Section~\ref{sec:radial1}.

\begin{proof}[Proof of Theorem~\ref{thm:RadialGen} when $D = \mathbb{D}$]
Let $\Pt_t$ denote the Markov semigroup generated by $(\theta_t,r_t)$ in \eqref{eq:SDEdisk}. We need to show that there exist $\epsilon, C > 0$ such that for every $r_0 \in (0,1]$ and $\theta_0, \tilde{\theta}_0 \in [0,2\pi)$ we have 
\begin{equation} \label{eq:diskgoal}
   \|\delta_{(r_0, \theta_0)}\Pt_T - \delta_{(r_0, \tilde{\theta}_0)}\Pt_T\|_{TV} \le 1-\epsilon
\end{equation}
for some $T \le C t_\nu(r_0)$. For a constant $C_0 \gg 1$ to be chosen below, we will consider separately the cases $r_0 \le 1- C_0 \nu^{1/3}$ and $r_0 \in (1-C_0\nu^{1/3}, 1]$. 

\textbf{Case 1}: Suppose that $r_0 \le 1 - C_0 \nu^{1/3}$. We first claim that if $C_0$ is large enough, then in this case we have 
\begin{equation} \label{eq:boundaryaway}
    10 \ell_\nu(r_0) \le \frac{1 - r_0}{2}
\end{equation}
for all $\nu$ sufficiently small. Indeed, by Assumption~\ref{ass:R2_general} we have $\lim_{r \to 1} b'(r) \neq 0$ and hence there exist $\delta > 0$ and $C_1 \ge 1$ such that for all $r \in [1-\delta,1]$ we have 
\begin{equation} \label{eq:nu13equiv}
C_1^{-1} \nu^{1/3} \le t_\nu(r) \le C_1 \nu^{1/3} \quad \text{and} \quad C_1^{-1} \nu^{1/3} \le \ell_\nu(r) \le C_1 \nu^{1/3}. 
\end{equation}
For $r_0 \in [1-\delta,1-C_0\nu^{1/3})$, the bound \eqref{eq:boundaryaway} follows by \eqref{eq:nu13equiv} provided that $C_0 > 20C_1$. If instead $r_0 < 1-\delta$, then \eqref{eq:boundaryaway} clearly holds for all $\nu$ sufficiently small due to $\lim_{\nu \to 0} \sup_{r \in [0,1]} \ell_\nu(r) = 0$.

With \eqref{eq:boundaryaway} established, we now fix $\theta_0, \tilde{\theta}_0 \in [0,2\pi)$ and seek to prove \eqref{eq:diskgoal}. First, we extend $b:[0,1] \to \R$ to a smooth, globally Lipschitz function $\bar{b}:[0,\infty) \to \R$. Then, for adapted controls $U_t$ and $V_t$, we let $\xi,\tilde{\xi}:[0,\infty) \times \R^2 \times \Omega \to \R^2$ denote the flow maps of the SDEs in \eqref{eq:radialSDEs} with $b$ replaced by $\bar{b}$ (recall these SDEs are viewed as equations on $\R^2$). Writing $\xi_r$ and $\xi_\theta$ for the radial and angular components of $\xi$, we further define 
$$ \tau(\omega) = \inf\left\{t \ge 0: \xi_r(t,r_0,\theta_0,\omega) = \frac{1+r_0}{2}\right\}.$$
Due to the stopping time $\bar{\tau}$ in our control argument from the
proof in Section~\ref{sec:radial1}, the set of random trajectories
that we considered to achieve \eqref{eq:radialhit} were such that $r_t
\le r_0 + 6\ell_\nu(r_0)$ for all $t \in [0,T]$. Therefore, due to
\eqref{eq:boundaryaway}, our proof in the $\R^2$ case implies that
there exist a choice of controls $U_t, V_t$ and constants $C,c > 0$
that do not depend on $r_0, \theta_0, \tilde{\theta_0}$, or $\nu$ such
that \eqref{eq:radial_cost} holds and for some $T \le Ct_\nu(r_0)$ there is $\Omega_0 \subseteq \Omega$ satisfying
\begin{equation} \label{eq:Omega0boundary}
    \P(\Omega_0) \ge c \quad \text{and} \quad \Omega_0 \subseteq \{\omega \in \Omega: \xi(T,r_0,\theta_0,\omega) = \tilde{\xi}(T,r_0,\tilde{\theta_0},\omega) \text{ and }\tau > T\}.
\end{equation}
Now define the measures $\lambda_1$ and $\lambda_2$ on $\mathbb{D}$ by
$$ \lambda_1(A) = \E \mathbf{1}_{\Omega_0}\mathbf{1}_A(\xi(T,r_0,\theta_0,\;\cdot\;)) \quad \text{and} \quad \lambda_2(A) = \E \mathbf{1}_{\Omega_0}\mathbf{1}_A(\xi(T,r_0,\tilde{\theta}_0,\;\cdot\;)) $$
and let $\Xi:[0,\infty) \times [0,1] \times [0,2\pi) \times \Omega \to [0,1]\times [0,2\pi)$ denote the flow map of \eqref{eq:SDEdisk}. Since $\Omega_0 \subseteq \{\tau > T\}$, for almost every $\omega \in \Omega_0$ we have $\Xi(t,r_0,\theta_0,\omega) = \xi(t,r_0,\theta_0,\omega)$ and $\Xi(t,r_0,\tilde{\theta_0},\omega) = \xi(t,r_0,\tilde{\theta_0},\omega)$ for all $t \in [0,T]$. We thus have $\delta_{(r_0,\theta_0)}\Pt_T \ge \lambda_1$ and $\delta_{(r_0,\tilde{\theta}_0)}\Pt_T \ge \lambda_2$. Since we also have \eqref{eq:Omega0boundary}, the bound \eqref{eq:diskgoal} now follows from the proof of Proposition~\ref{thm:Girsanov}.

\textbf{Case 2}: With $C_0$ fixed as in Case 1 we now suppose that $r_0 \in (1-C_0\nu^{1/3},1]$. In this case, by \eqref{eq:nu13equiv}, it suffices to prove \eqref{eq:diskgoal} for some $T \le C\nu^{-1/3}$. Moreover, by Case 1, there exist $\epsilon_0 > 0$ and $T_0 \leqc \nu^{-1/3}$ such that 
\begin{equation} \label{eq:case1result}
    \sup_{\theta, \tilde{\theta} \in [0,2\pi)}\|\delta_{(r,\theta)}\Pt_{T_0} - \delta_{(r,\tilde{\theta})}\Pt_{T_0}\|_{TV} \le 1- \epsilon_0 \quad \forall r \in [1-3C_0\nu^{1/3}, 1-C_0\nu^{1/3}].
\end{equation}
Let $A = [1-3C_0\nu^{1/3}, 1- C_0 \nu^{1/3}] \times [0,2\pi)$.
By Lemma~\ref{lem:escape} applied with $\ell = C_0\nu^{1/3}$, there exist $C_2, c_2 > 0$ that do not depend on $r_0$ or $\nu$ and $T_1 \le C_2 \nu^{-1/3}$ such that
$$ \delta_{(r_0,\theta_0)}\Pt_{T_1}(A) \ge c_2 \quad \text{and} \quad \delta_{(r_0,\tilde{\theta_0})}\Pt_{T_1}(A) \ge c_2.  $$
Let $T = T_0 + T_1$ and define the measures $\mu_1 = \delta_{(r_0,\theta_0)}\Pt_{T_1}$ and $\mu_2 = \delta_{(r_0, \tilde{\theta}_0)}\Pt_{T_1}$. We also define the restricted measures $\mu_j^{(A)}(\cdot) = \mu_j(A \cap \cdot)$. We then have 
\begin{align*} 
\|\delta_{(r_0,\theta_0)}\Pt_{T} - \delta_{(r_0,\tilde{\theta}_0)}\Pt_{T}\|_{TV} &= \|\mu_1 \Pt_{T_0} - \mu_2 \Pt_{T_0}\|_{TV} \\ 
& \le  \|\mu^{(A)}_1\Pt_{T_0} - \mu_2^{(A)}\Pt_{T_0}\|_{TV} + \|\mu_1 - \mu_1^{(A)}\|_{TV} + \|\mu_2 - \mu_2^{(A)}\|_{TV}.\\
&\leq \|\mu^{(A)}_1\Pt_{T_0} - \mu_2^{(A)}\Pt_{T_0}\|_{TV}  +
                                                                                                                               \tfrac12\big(\mu_1(A^c)
                                                                                                                               +\mu_2(A^c)\big)
\end{align*}
Since $\mu^{(A)}_1$ and $\mu^{(A)}_2$ have the same $r$-marginal
distributions, it follows by Proposition~\ref{thm:marginal measures} and \eqref{eq:case1result} that 
$$ \|\mu^{(A)}_1\Pt_{T_0} - \mu_2^{(A)}\Pt_{T_0}\|_{TV} \le (1-\epsilon_0)\|\mu_1^{(A)} - \mu_2^{(A)}\|_{TV} \le (1-\epsilon_0)\mu_1(A). $$
Combining this bound with the previous ones and the fact that
$\mu_1(A^c)=\mu_2(A^c)=1-\mu_1(A)$, we obtain 
$$ \|\delta_{(r_0,\theta_0)}\Pt_{T} - \delta_{(r_0,\tilde{\theta}_0)}\Pt_{T}\|_{TV}\le (1-\epsilon_0)\mu_1(A) + 1-\mu_1(A) \le 1-\epsilon_0\mu_1(A) \le 1- \epsilon_0 c_2. $$
\end{proof}
  
\section{Proof of local enhanced dissipation estimates on $\T^2$} \label{sec:local}    
Let $b\colon\T \to \R$ denote a function satisfying Assumption~\ref{Ass:classical}, $t_\nu\colon\T \to (0,\infty)$ be the associated local timescale defined in \eqref{eq:tnuintro}, and $\Pt_t$ the Markov semigroup generated by \eqref{eq:SDE2} taken as an equation on $\T^2$. The proof of Theorem~\ref{thm:T2general} given earlier shows that there exist $C_0 \ge 1$ and $\epsilon \in (0,1)$ such that for any $y \in \T$ we have 
\begin{equation} \label{eq:thm1consequence}
\sup_{x,\tilde{x}}\|\delta_{(x,y)}\Pt_t - \delta_{(\tilde{x},y)}\Pt_t\|_{TV} \le 1-\epsilon, \quad \text{where} \quad t = C_0 t_\nu(y).
\end{equation}
In this section, we will iterate \eqref{eq:thm1consequence} by using Proposition~\ref{thm:marginal measures} to prove Theorem~\ref{thm:T2local}. To this end first observe that, in view of \eqref{eq:streamlineheuristic}, \eqref{eq:tnuAss2}, and Corollary~\ref{cor:finitevanish}, to prove Theorem~\ref{thm:T2local} it is sufficient to show that there exist constants $C,c > 0$ such that for every $y \in \T$ we  have 
\begin{equation}\label{eq:localgoal}
    \sup_{x, \tilde{x} \in \T}\|\delta_{(x,y)}\Pt_T - \delta_{(\tilde{x},y)}\Pt_T\|_{TV} \le C \nu^{c|\log \nu|} \quad \text{for some} \quad T \le C(1+|\log \nu|^{4N})t_\nu(y),
\end{equation}
where $N$ denotes the maximal order of vanishing of $b'$ at the
critical points of $b$. Our goal in the remainder of the section is
thus to prove \eqref{eq:localgoal}. 

We start with two simple lemmas. The first one is essentially just an estimate for the 1d periodized heat kernel and will be used to control how much mass of the measure $\delta_{(x,y)}\Pt_t - \delta_{(\tilde{x},y)}\Pt_t$ has diffused away from the streamline containing $(x,y)$.

\begin{lemma} \label{lem:heat}
    Fix $y \in \T$ and $R \ge 1$. Let $K_t = \{y' \in \T: |y-y'| \ge R\sqrt{\nu t}\}$, where $|\cdot|$ denotes the distance function on $\T$. Then, for any $x \in \T$ and $t \ge 0$ we have 
    $$ \delta_{(x,y)}\Pt_t(\T \times K_t) \le \sqrt{2}e^{-R^2/8}. $$
\end{lemma}

\begin{proof}
    We have that 
    \begin{equation}\label{eq:heat_representation}
        \delta_{(x,y)} \Pt_t(\T \times K_t) = \int_{K_t} \Phi_y(t,y')\dee y',
    \end{equation}
    where $\Phi_y(t,y')$ is the fundamental solution of the heat equation on $\T$ starting from the point $y$. By comparing with the solution to the 1d heat equation on the whole space, we see that 
    $$ \int_{\T \setminus K_t} \Phi_y(t,y')\dee y' \ge \frac{1}{\sqrt{4\pi \nu t}}\int_{|z|< R\sqrt{\nu t}} e^{-\frac{z^2}{4\nu t}} \dee z = 1 - \frac{1}{\sqrt{\pi}} \int_{|z| \ge R/2} e^{-z^2}\dee z \ge 1 - \sqrt{2}e^{-R^2/8}.$$
    Thus, 
    $$ \int_{K_t} \Phi_y(t,y')\dee y' = 1 - \int_{\T\setminus K_t} \Phi_y(t,y')\dee y' \le \sqrt{2}e^{-R^2/8},$$
    which completes the proof.
\end{proof}

Next, we have a lemma that says the local timescale $t_\nu$ is roughly
constant on suitably sized spatial intervals, provided that one is not too close to the critical points.

\begin{lemma}\label{lem:timescalerelation}
Let $\{y_i\}_{i=1}^M$ denote the critical points of a function $b\colon \T \to \R$ satisfying Assumption~\ref{Ass:classical} and let $k_i \ge 1$ be the order of vanishing of $b'$ at the point $y_i$. Let $\ell_{\nu} = \sqrt{\nu t_\nu}$ and for each $y \in \T$ define the interval 
$$ I(y) = \big(y - |\log \nu|^2 \ell_{\nu}(y), y + |\log \nu|^2 \ell_{\nu}(y)\big). $$
There is a constant $C \ge 1$ such that for all $\nu$ sufficiently small and $y \not \in \cup_{i=1}^M I(y_i)$ we have
$$ \sup_{y' \in I(y)} t_\nu(y') \le C t_\nu(y).$$
\end{lemma}

\begin{proof}
    Here, for two positive numbers $a$ and $b$ we will write $a \approx b$ to mean that $C^{-1} a \le b \le Cb$ for some $C \ge 1$ that does not depend on $\nu$. Due to the relationship between $t_\nu$ and $\ell_\nu$, it is sufficient to show that for some $C \ge 1$ and all $y \not \in \cup_{i=1}^M I(y_i)$ we have
    \begin{equation}\label{eq:ellnutechnical1}
        \sup_{y' \in I(y)} \ell_\nu(y') \le C \ell_\nu(y).
    \end{equation}
    From the discussion preceding the statement of Theorem~\ref{thm:T2local}, we know that 
\begin{equation}\label{eq:ellnutechnical2}
    \ell_\nu(y) \approx \nu^{1/3}|b'(y)|^{-1/3}
    \end{equation} 
    whenever $y \not \in \cup_{i=1}^M I(y_i)$. Due to \eqref{eq:ellnutechnical2} and $\lim_{\nu \to 0} \sup_{y \in \T} |\log \nu|^2 \ell_\nu(y) = 0$, for any fixed $0 \le \delta \ll 1$ the bound \eqref{eq:ellnutechnical1} clearly holds if $y$ is such that $|y-y_i| \ge \delta$ for every $1 \le i \le M$.  Suppose now instead that $|\log \nu|^2 \ell_\nu(y_i)\le |y-y_i| \le \delta$ for some $i$. Assuming $\delta$ is small enough we have
    \begin{equation} \label{eq:ellnutechnical3}
    \ell_\nu(y) \approx \nu^{1/3}|y-y_i|^{-\frac{k_i}{3}}. 
    \end{equation}
    Since $\ell_\nu(y_i) \approx \nu^{\frac{1}{k_i+3}}$, it follows from \eqref{eq:ellnutechnical3} that 
$$ \ell_\nu(y) \le C_1 |\log \nu|^{-\frac{2k_i}{3}} \ell_\nu(y_i) $$
    for some $C_1 > 0$. It is clear then that for all $\nu$ sufficiently small we have 
    $$ \inf_{y' \in I(y)}|y'-y_i| \approx |y-y_i|, $$
    so that \eqref{eq:ellnutechnical1} follows from \eqref{eq:ellnutechnical3}.
\end{proof}

We are now ready to give the proof of Theorem~\ref{thm:T2local}. In the following, for a (possibly signed) measure $\mu$ on $\T^2$ and a measurable set $A \subseteq \T^2$ we write $\mu_A$ for its restriction to $A$, defined by 
\begin{equation} \label{eq:signedmeasure}
\mu_A(B) = \mu(A \cap B)
\end{equation}
for all $B \in \mathcal{B}(\T^2)$.

\begin{proof}[Proof of Theorem~\ref{thm:T2local}]
Recall that it suffices to prove \eqref{eq:localgoal}. Fix $x, \tilde{x}, y \in \T$. We will split the proof into cases depending on whether or not $y \in \cup_{i=1}^M I(y_i)$, where $y_i$ and $I(y_i)$ are both as in Lemma~\ref{lem:timescalerelation}.

\vspace{0.2cm}
\noindent \textbf{Case 1}: Suppose that $y \not \in \cup_{i=1}^M I(y_i)$. Let $A = \T \times I(y)$ and $t_\nu^* = C_0\sup_{y' \in I(y)}t_\nu(y')$, where $C_0$ is as in \eqref{eq:thm1consequence}. Let $\{\mu_i\}_{i \ge 0}$ and $\{\eta_i\}_{i \ge 0}$ denote the sequences of measures defined by
\begin{equation}\label{eq:expanding_measures_inside}
        \mu_i = 
        \begin{cases}
            \delta_{(x, y)} - \delta_{(\tilde{x}, y)} & i = 0 \\
            (\mu_{i - 1}\mathcal{P}_{t_\nu^*})_{A} & i \geq 1
        \end{cases}
        \quad \text{and} \quad 
        \eta_i = 
        \begin{cases}
            0 & i = 0 \\
            (\mu_{i - 1}\mathcal{P}_{t_\nu^*})_{A^c} & i \geq 1,
        \end{cases}
    \end{equation}
    where we have used the notation introduced in \eqref{eq:signedmeasure}. The definitions above are such that for $i \ge 1$ we have $\mu_{i-1} \Pt_{t_\nu^*} = \mu_i + \eta_i$
and it then follows easily that for any $n \ge 1$ there holds
\begin{equation}\label{eq:local1}   \delta_{(x,y)}\mathcal{P}_{nt_\nu^\ast}- \delta_{(\tilde{x},y)}\mathcal{P}_{nt_\nu^\ast} = \mu_n + \sum_{i = 1}^n\eta_i\mathcal{P}_{(n - i)t_\nu^\ast}.
    \end{equation}
By construction, for each $i \ge 0$, $\mu_i$ is a signed measure with $\mu_i(\T^2) = 0$, zero $y$-marginal distribution, and $\|\mu_i\|_{TV(A^c)} = 0$. Therefore, by \eqref{eq:thm1consequence} and Proposition~\ref{thm:marginal measures}, for every $i \ge 1$ we have 
\begin{equation}
  \|\mu_{i}\|_{TV} \le \|\mu_{i-1} \Pt_{t_\nu^*}\|_{TV} \le (1-\epsilon)\|\mu_{i-1}\|_{TV}. 
\end{equation}
Let $n_0$ be the first natural number greater than or equal to $|\log \nu|^2$. Iterating the bound above for $1 \le i \le n_0$ and using $\|\mu_0\|_{TV} = 1$ we obtain 
\begin{equation}\label{eq:local2}
        \|\mu_{n_0}\|_{TV} \leq (1 - \epsilon)^{n_0} \le \nu^{c_1|\log \nu|}
    \end{equation}
    for some $c_1 > 0$. To complete the proof, we still need to bound the sum in \eqref{eq:local1}. We will show that for some $C_1 > 0$ we have  \begin{equation}\label{eq:eta_bound}
    \|\eta_i\|_{TV} \leq 2\sqrt{2} e^{-\frac{1}{8C_1} |\log \nu|^2}
\end{equation}
for every $1 \le i \le n_0$. Accepting this estimate for the moment, putting \eqref{eq:eta_bound} and \eqref{eq:local2} into \eqref{eq:local1} applied with $n = n_0$ we find 
\begin{equation} \label{eq:localfinal} \|\delta_{(x,y)}\mathcal{P}_{n_0t_\nu^\ast}- \delta_{(\tilde{x},y)}\mathcal{P}_{n_0t_\nu^\ast}\|_{TV} \leq C|\log\nu|^2\nu^{c|\log\nu|} 
\end{equation}
for some constants $C,c > 0$. The fact that \eqref{eq:localfinal} implies \eqref{eq:localgoal} follows from Lemma~\ref{lem:timescalerelation} and the definitions of $n_0$ and $t_\nu^*$.

All that remains now is to prove \eqref{eq:eta_bound}. We define the operator $\Theta$ to act on a Borel probability measure on $\T^2$ via $\Theta\mu = (\mu\mathcal{P}_{t_\nu^\ast})_A$. Then, for $1 \leq i \leq n_0$ we have 
\begin{equation}\label{eq:eta_representation}
    \eta_i = ((\Theta^{i - 1}\mu_0)\mathcal{P}_{t_\nu^\ast})_{A^c}.
\end{equation}
Since $\Theta$ is linear, we see that 
\begin{equation}\label{eq:eta_tv_bound}
    \|\eta_i\|_{TV} \leq 2\| (\Theta^{i - 1}\delta_{(x, y)})\mathcal{P}_{t_\nu^\ast} \|_{TV(A^c)}.
\end{equation}
For any nonnegative measure $\mu$ we have $\Theta \mu \le \mu \Pt_{t_\nu^*}$ and hence
\begin{equation} \label{eq:comparison}
\|(\Theta\mu)\mathcal{P}_{t_\nu^\ast}\|_{TV(A^c)} \leq \| \mu \mathcal{P}_{2t_\nu^\ast} \|_{TV(A^c)}. 
\end{equation}
Using \eqref{eq:comparison} in \eqref{eq:eta_tv_bound}, for any $1\le i \le n_0$ we have 
\begin{equation}\label{eq:etai_tv_2}
    \|\eta_i\|_{TV} \leq 2\| \delta_{(x, y)}\mathcal{P}_{i t_\nu^\ast} \|_{TV(A^c)}.
\end{equation}
Observe now that by Lemma~\ref{lem:timescalerelation} we have $t_\nu^* \le C_1 t_\nu(y)$ for some $
C_1 \ge 1$, and hence for every $1\le i \le n_0$ we have that 
$$ \sqrt{\nu i t_\nu^*} \le \sqrt{C_1}|\log \nu| \ell_\nu(y).$$
Thus, by Lemma~\ref{lem:heat} and the definition of $I(y)$ we have
\begin{equation} \label{eq:eta_tv_3}\|\delta_{(x,y)}\Pt_{it_\nu^*}\|_{TV(A^c)} \le \sqrt{2} e^{-\frac{1}{8C_1} |\log \nu|^2} 
\end{equation}
for every $1\le i \le n_0$.
Putting \eqref{eq:eta_tv_3} into \eqref{eq:etai_tv_2} completes the proof of \eqref{eq:eta_bound}.

\vspace{0.2cm}
\noindent \textbf{Case 2}: Next suppose that $y \in I(y_i)$ for some $1 \le i \le M$. By carrying out the same argument as in Case 1 but with $I(y)$ replaced by $\{y' \in \T: |y-y'| \le |\log \nu|^2 \ell_\nu(y_i)\}$ and $t_\nu^*$ replaced by $t_\nu(y_i)$, one can show that for some $n \le 2|\log \nu|^2$ and $c > 0$ we have 
\begin{equation} \label{eq:local5}
\|\delta_{(x,y)}\Pt_{n t_\nu(y_i)} - \delta_{(\tilde{x},y)}\Pt_{n t_\nu(y_i)}\|_{TV} \le \nu^{c|\log \nu|}  \end{equation}
for all $\nu$ sufficiently small. Now, recalling \eqref{eq:tnuAss2} and \eqref{eq:localrate}, it is easy to see that $$t_\nu(y_i)\approx \nu^{-\frac{k_i+1}{k_i+3}}\quad  \text{and} \quad t_\nu(y) \gtrsim |\log \nu|^{-\frac{4k_i}{3}} \nu^{-\frac{k_i+1}{k_i+3}}.$$
Thus, for some $C \ge 1$ we have 
$$ n t_\nu(y_i) \le C|\log\nu|^{\frac{4k_i}{3}+2}t_\nu(y),$$
and hence \eqref{eq:local5} implies \eqref{eq:localgoal}.
 
\end{proof}

\appendix 
\section{Proofs of lemmas on length and timescale relations} \label{appendix}

In this appendix, we collect the proofs of Lemmas~\ref{lem:technical1}, \ref{lem:technical2}, and \ref{lem:technicalradial}.

\begin{proof}[Proof of Lemma~\ref{lem:technical1}]
We begin with the proof of \eqref{eq:b'equivalence}. For $\epsilon > 0$, let $I_i(\epsilon) = (y_i - \epsilon, y_i + \epsilon)$ for $m\le i \le M$ and $\bar{I}_i(\epsilon) = (\bar{y}_i - \epsilon, \bar{y}_i +\epsilon)$ for $1 \le i \le N$. We decompose each punctured interval $\bar{I}_i(\epsilon) \setminus \{\bar{y}_i\}$ as 
$$ \bar{I}_i(\epsilon) \setminus \{\bar{y}_i\} = (\bar{y}_i-\epsilon,\bar{y}_i) \cup (\bar{y}_i, \bar{y}_i+\epsilon) : = \bar{I}_i^-(\epsilon) \cup \bar{I}_i^+(\epsilon).$$
By Assumption~\ref{Ass:T2general}a, there exists $\epsilon' > 0$ so that $b'$ is monotone in any interval $\bar{I}_i^\pm(\epsilon')$ in which $|b'|$ is unbounded. Defining $S = \T \setminus \cup_{i=1}^{m-1}I_i(\delta)$, we will consider separately the cases
\begin{equation}\label{eq:b'equivcases}
y \in S \setminus \Big(\bigcup_{i=m}^{M} I_i(\epsilon) \cup \bigcup_{i=1}^N \bar{I}_i(\epsilon)\Big), \quad y \in \bigcup_{i=m}^{M} I_i(\epsilon), \quad \text{and} \quad y \in \bigcup_{i=1}^N \bar{I}_i(\epsilon),
\end{equation}
where $\epsilon > 0$ will be chosen at least small enough so that $\epsilon \ll \epsilon'$ and $\nu$ is assumed small enough so that $\sup_{y\in \T}\ell_\nu(y) \ll \epsilon\wedge\delta$.
In the first case of \eqref{eq:b'equivcases}, the bound \eqref{eq:b'equivalence} with some $C > 0$ depending only on $\epsilon$ and $b$ follows immediately by the extreme value theorem. 
Consider next the case where $y \in I_i(\epsilon)$ for some $m \le i \le M$. Without loss of generality, we may assume that $y \in (y_i,y_i+\epsilon)$, which implies that $\iota(y) = 1$. Let $k \ge 2$ denote the minimal natural number such that $b^{(k)}(y_i) \neq 0$. By Taylor's theorem, for any $z \in (y,y+4\ell_\nu(y))$ we have 
$$ b'(z) = \frac{b^{(k)}(y_i)}{(k-1)!}(z-y_i)^{k-1}+\frac{b^{(k+1)}(\xi)}{k!}(z-y_i)^k $$
for some $\xi \in (y_i,z)$. It follows that for $\epsilon$ sufficiently small we have  
$$ \frac{1}{2}\Big|\frac{b^{(k)}(y_i)}{(k-1)!}\Big||z-y_i|^{k-1} \le |b'(z)| \le 2\Big|\frac{b^{(k)}(y_i)}{(k-1)!}\Big||z-y_i|^{k-1}. $$
One can use this to show that for any $h \in [1/2,4]$, one has 
$$ 4^{-k} \le  \frac{|b'(y+h\ell_\nu(y))|}{|b'(y+\ell_\nu(y))|} \le 4^k, $$
as desired. For the final case of \eqref{eq:b'equivcases}, it suffices to assume that $y \in \bar{I}^-_i(\epsilon)$ and $z \mapsto b'(z)$ increases monotonically to $\infty$ as $z\uparrow \bar{y}_i$. Then, $\iota(y) = -1$ and to prove \eqref{eq:b'equivalence} we must only show that there exists $C > 0$ that does not depend on $y$ such that 
\begin{equation} \label{eq:b'equivcase3goal}
    \frac{|b'(y-\frac{1}{2}\ell_\nu(y))|}{|b'(y-4 \ell_\nu(y))|} \le C.
\end{equation}
Define $F:[1/2,4] \to \R$ by 
$$F(h) = b'(y-h\ell_\nu(y)).$$
By the second estimate in Assumption~\ref{Ass:T2general}a, there exists $p > 0$ such that 
$$F'(h) = -\ell_\nu(y)b''(y-h\ell_\nu(y)) \ge - \frac{p\ell_\nu(y)}{y_i-y+h\ell_\nu(y)}F(h) \ge -\frac{p}{h}F(h). $$
Thus,  $F(h) \ge F(1/2)(2h)^{-p}$, which implies that \eqref{eq:b'equivcase3goal} holds with $C = 8^{p}$.

Next, we will prove \eqref{eq:tnuboundregular}. First, observe that by the definitions of $t_\nu$ and $\ell_\nu$, \eqref{eq:tnuboundregular} is equivalent to the bounds 
\begin{equation}\label{eq:tnuboundequiv}
   \frac{1}{C} \ell_\nu(y)|b'(y+\iota(y)\ell_\nu(y))| \le |b(y+\iota(y)\ell_\nu(y)) - b(y)| \le C\ell_\nu(y)|b'(y+\iota(y) \ell_\nu(y))|.
\end{equation}
The first inequality above is an immediate consequence of \eqref{eq:b'equivalence}. Indeed, using the fact that the sign of $b'$ does not change over the interval between $y$ and $y+\iota(y)\ell_\nu(y)$, it follows from \eqref{eq:b'equivalence} that for some $c > 0$ we have 
$$ |b(y+\iota(y)\ell_\nu(y)) - b(y)| \ge \frac{\ell_\nu(y)}{2} \min_{h \in [1/2,1]}|b'(y+\iota(y)\ell_\nu(y)h)| \ge c\ell_\nu(y)|b'(y+\iota(y)\ell_\nu(y))|.$$
We now turn to the second inequality in \eqref{eq:tnuboundequiv}. In both the first case of \eqref{eq:b'equivcases} and when $y$ is in a half interval $\bar{I}_i^\pm$ where $b'$ remains bounded, the inequality is immediate from the fundamental theorem of calculus and the extreme value theorem. In the second case of \eqref{eq:b'equivcases}, the estimate follows easily from the fact that $b''(z)$ has a fixed sign in each of the intervals $(y_i,y_i+2\epsilon)$ and $(y_i-2\epsilon,y_i)$ provided that $\epsilon$ is sufficiently small. It remains to consider the case where $y \in I_i^{\pm}(\epsilon)$ and the half interval is such that $b'$ becomes unbounded. Suppose without loss of generality that $y \in \bar{I}_i^+(\epsilon)$ and $z \mapsto b'(z)$ increases monotonically to $\infty$ as $z \downarrow \bar{y}_i$. If $y-y_i \le \ell_\nu(y)$, then by the first estimate in Assumption~\ref{Ass:T2general}a we have 
$$ |b(y+\ell_\nu(y)) - b(y)|\le |b(y+\ell_\nu(y)) - b(y_i)| \le C|y-y_i + \ell_\nu(y)||b'(y+\ell_\nu(y))| \le 2C\ell_\nu(y)|b'(y+\ell_\nu(y))|. $$
If instead $y - y_i > 2\ell_\nu$, then an argument similar to the proof of \eqref{eq:b'equivcase3goal} shows that there exists $c > 0$ such that 
$$|b'(y+\ell_\nu(y))| \ge c|b'(z)| \quad \forall z \in [y,y+\ell_\nu(y)],$$
from which the second inequality of \eqref{eq:tnuboundequiv} follows immediately. 
\end{proof}

\begin{proof}[Proof of Lemma~\ref{lem:technical2}]
    We may assume without loss of generality that $y \in (y_i,y_i+\delta)$, so that $\iota(y) = 1$. As $h \mapsto |b'(y_i + h)|$ is monotone increasing for $h \in [0,2\delta]$ and 
$$t_\nu(y) = t_\nu(y_i) = \frac{1}{|b(y_i+\ell_\nu(y_i)) - b(y_i)|}$$
by definition, we have 
\begin{align*} 
\frac{1}{\sqrt{\nu}|b'(y+C\ell_\nu(y))|} & \le \frac{1}{\sqrt{\nu}|b'(y_i+C\ell_\nu(y_i))|} \\ 
& = \Big(\frac{|b(y_i+ \ell_\nu(y_i)) - b(y_i)|}{\sqrt{\nu}|b'(y_i+C\ell_\nu(y_i))|}\Big) t_\nu(y_i) \\
& = \Big(\frac{\sqrt{|b(y_i+ \ell_\nu(y_i)) - b(y_i)|}}{\ell_\nu(y_i)|b'(y_i+C\ell_\nu(y_i))|}\Big) t_\nu(y).
\end{align*}
The quantity in parentheses above remains bounded as $\nu \to 0$ by Assumption~\ref{Ass:T2general}b, and so \eqref{eq:infiniteordertechnical} follows.
\end{proof}

\begin{proof}[Proof of Lemma~\ref{lem:technicalradial}]

Due to \eqref{eq:qdef}, we may assume without loss of generality that $\lim_{r \to 0}r^{-q}b'(r) = 1$. Moreover, by \eqref{eq:R2assumption} there exist $c_1,C_1 > 0$ and $R \ge 2$ such that 
\begin{equation} \label{eq:b'nearinf}
    |b'(r)| \ge c_1 \quad \text{and} \quad \frac{|b''(r)|}{|b'(r)|} \le C_1 
\end{equation}
for all $r \ge R-1$. Let $0 < \delta \ll h_0 \wedge \sup_{r\ge 0} \ell_\nu(r)$ be chosen small enough so that 
\begin{equation} \label{eq:b'nearzero}
\frac{r^q}{2} \le b'(r) \le 2 r^q 
\end{equation}
for all $r \in [0,2\delta)$ and define $I = [0,\delta)$. The proof in the case that $\delta \le r < R$ is already taken care of by Lemma~\ref{lem:technical1}. When $r \in I$, the estimate \eqref{eq:b'equivalencerad} is immediate from \eqref{eq:b'nearzero} and \eqref{eq:tnuboundregularrad} also follows just as easily from \eqref{eq:b'nearzero} and the fact that it is equivalent to the bound 
\begin{equation}\label{eq:tnuradequiv}
\frac{1}{C} \ell_\nu(r)|b'(r+\iota(r)\ell_\nu(r))| \le |b(r+\iota(r)\ell_\nu(r)) - b(r)| \le C \ell_\nu(r)|b'(r+\iota(r)\ell_\nu(r))|
\end{equation}
for some $C \ge 1$. We now suppose that $r \ge R$. Define $F:[0,1] \to \R$ by $F(h) = |b'(r+\iota(r)h)|$. By \eqref{eq:b'nearinf}, we have $F'(h) \le C_1 F(h)$ and hence it follows from Gr\"{o}nwall's lemma that
\begin{equation} \label{eq:tnuradequiv1}
  F(h) \le F(0)e^{C_1 h}.  
\end{equation}
The bound \eqref{eq:b'equivalencerad} is immediate from \eqref{eq:tnuradequiv1}, and \eqref{eq:tnuboundregularrad} follows from \eqref{eq:tnuradequiv1}, the equivalence of \eqref{eq:tnuboundregularrad}
and \eqref{eq:tnuradequiv}, and the fact that $\bar{r} \mapsto b'(\bar{r})$ does not change sign for $\bar{r} \in (R-1,\infty)$. \end{proof}

\addcontentsline{toc}{section}{References}
\bibliographystyle{abbrv}
\bibliography{shearbib}

\end{document}